\documentclass[11pt]{article}
\usepackage[margin=1in]{geometry}
\usepackage[authoryear,round]{natbib}

\title{A One-Sided McDiarmid-Type Inequality for Heavy Tails via Reverse KL Tensorization}
\author{
  Haruka Eshima \\
  \small Okinawa Institute of Science and Technology, Okinawa, Japan \\
  \small \texttt{haruka.eshima@oist.jp}
  \and
  Amedeo Roberto Esposito \\
  \small  Okinawa Institute of Science and Technology, Okinawa, Japan \\
  \small \texttt{amedeo.esposito@oist.jp}
  \and
  Andrea Paudice \\
  \small Department of Computer Science, Aarhus University, Aarhus, Denmark \\
  \small \texttt{apaudice@cs.au.dk}
}
\date{}
\usepackage{amsmath,amssymb,mathtools,bm}
\usepackage{booktabs}
\usepackage{xcolor}
\usepackage{comment}
\usepackage[colorlinks=true,linkcolor=blue,citecolor=blue,urlcolor=blue]{hyperref}
\usepackage[nameinlink]{cleveref}
\usepackage{amsthm}
\newtheorem{theorem}{Theorem}[section]
\newtheorem{lemma}[theorem]{Lemma}
\newtheorem{proposition}[theorem]{Proposition}
\newtheorem{corollary}[theorem]{Corollary}
\theoremstyle{remark}
\newtheorem{remark}[theorem]{Remark}
\theoremstyle{plain}
\newenvironment{proofof}[1]{\begin{proof}[Proof of #1]}{\end{proof}}
\newenvironment{keywords}{\par\noindent\textbf{Keywords:} }{\par}
\usepackage{enumitem}
\usepackage{colortbl}
\usepackage{multirow}
\usepackage{hhline}
\usepackage{threeparttable}
\definecolor{ourgreen}{rgb}{0.85,0.97,0.82}
\newcommand{\ours}{\rowcolor{ourgreen}}
\newcommand{\yesmark}{\textcolor{green!50!black}{$\checkmark$}}
\newcommand{\nomark}{\textcolor{red!75!black}{$\times$}}
\newcommand{\thinline}{\arrayrulecolor{black!45}\hhline{~-----}\arrayrulecolor{black}}
\usepackage{newtx}
\newcommand{\E}{\mathbb{E}}
\newcommand{\Pp}{\mathbb{P}}

\newcommand{\X}{\bm{X}}
\newcommand{\Y}{\bm{Y}}

\newcommand{\U}{\bm{U}}

\newcommand{\Dk}{D_k(\X)}
\newcommand{\KL}{D_{\mathrm{KL}}}
\newcommand{\KLi}[1]{\KL^{(i)}\!\left(#1\right)}
\newcommand{\esssup}{\operatorname*{ess\,sup}}

\newcommand{\med}{\operatorname{med}}
\newcommand{\tcell}[1]{\begin{tabular}[t]{@{}l@{}}#1\end{tabular}} 
\newcommand{\ccell}[1]{\begin{tabular}[t]{@{}c@{}}#1\end{tabular}} 

\providecommand{\Prob}{\mathbb{P}}

\newtheorem{assumption}{Assumption}[section]
\crefname{assumption}{Assumption}{Assumptions}
\Crefname{assumption}{Assumption}{Assumptions}

\begin{document}

\maketitle

\begin{abstract}
We revisit \emph{McDiarmid's inequality}, a central tool in statistical learning theory and algorithm analysis. We show that large deviations above the mean are exponentially unlikely for functions of independent random variables that are almost surely bounded above and satisfy a conditional moment assumption of order $p\in(1,2]$. When $p<2$, the result applies even when the function has infinite variance. Our proof relies on a new tensorization inequality for the Kullback--Leibler divergence, termed \emph{reverse KL tensorization}, which may be of independent interest. Applications include concentration of the sample mean and median-of-means estimator, and one-sided uniform risk bounds for unbounded losses.
\end{abstract}


\section{Introduction}
Concentration inequalities are fundamental tools in statistical learning theory, multi-armed bandits, and randomized algorithms. Among them, McDiarmid's inequality \citep{mcdiarmid1989method} has been extensively studied and applied \citep{kutin2002extensions, kutin2002almost, kontorovich2014concentration, maurer2021concentration, li2023distribution, li2024concentration, li2024concentration_icml, lei2026stability}. Its distribution-free guarantees make it particularly useful for analyzing the generalization error of learning algorithms \citep{kontorovich2014concentration, bousquet2002stability, boucheron2005theory, bousquet2020sharper}. Moreover, it applies directly to nonlinear statistics \citep{maurer2018weak, maurer2019bernstein, maurer2019uniform, li2024concentration}, extending beyond the setting of sums of independent random variables covered by the classical Bernstein inequality \citep{wainwright2019high}. However, McDiarmid's inequality requires the bounded-differences condition: replacing any single input, while keeping the others fixed, can change the function's value by at most a fixed amount, regardless of the input values (see \Cref{sec:Related Work}). This requirement can limit its direct applicability to learning problems with unbounded loss functions \citep{kontorovich2014concentration, li2024concentration}.

This limitation is particularly relevant in applications where data exhibit heavy tails \citep{cont2001empirical, Newman01092005, clauset2009, piantadosi2014zipf, zhang2020adaptive, kunstner2024heavy}. The loss itself may also be heavy-tailed \citep{brownleesEmpiricalRiskMinimization2015, hsu2016loss, cortes2019relative}, potentially with infinite variance \citep{aminiangeneralization}. Such settings motivate the development of concentration inequalities beyond boundedness and sub-Gaussian assumptions \citep{Catoni2012, mendelsonLearningConcentration2015, simsekliTailIndexAnalysisStochastic2019, martin2021implicit}. Although robust mean estimators can achieve sub-Gaussian deviation rates under a finite-variance assumption \citep{Catoni2012, devroye2016subgaussian}, these rates cannot generally be guaranteed under weaker moment assumptions \citep{bubeck2013bandits, devroye2016subgaussian}. The empirical mean is particularly sensitive to heavy tails: its deviation probabilities may decay only polynomially with the sample size \citep{FukNagaev1971, Nagaev1979, Catoni2012, bakhshizadeh2023sharp}.

In many machine learning and statistical applications, however, the quantities of interest are non-negative losses or returns \citep{bubeck2013bandits, thomas2015high, mendelsonLearningConcentration2015, koltchinskii2015bounding, oliveira2016lower}. This structure can yield stronger one-sided concentration guarantees. Recently, \citet{aminiangeneralization} used non-negativity to extend the one-sided Bernstein inequality \citep[Proposition 2.14]{wainwright2019high} to the heavy-tailed setting under a finite $p$-th moment assumption, for $p\!\in\!(1,2]$. The resulting inequality exponentially bounds the probability that the empirical mean falls below its expectation. For losses, this controls the upper tail of the population risk minus the empirical risk, showing that robust estimators are not always necessary for one-sided generalization guarantees.

\paragraph{Contribution.}In this work, we establish a heavy-tailed extension of McDiarmid's inequality for one-sided deviations. Our result provides exponential concentration for a class of nonlinear functions, extending beyond the additive setting considered by \citet{aminiangeneralization}. To address the challenges of deriving concentration bounds for nonlinear functions, we develop a new reverse KL tensorization argument (\Cref{thm:reverse-kl}) that may be of independent interest. We illustrate the scope of our inequality through three applications: recovering existing bounds for additive functions, establishing concentration of the median-of-means (MoM) estimator \citep{nemirovskij1983problem, jerrum1986random, alon1996space} around its own expectation, and bounding the population risk uniformly over a class of nonnegative, possibly unbounded
losses. In contrast to the finite-class uniform bound of \citet[Theorem~5.4]{aminiangeneralization}, our uniform risk bound applies to possibly infinite classes with finite Rademacher complexity.

\begin{table}[t]
  {\caption{\small Comparison with previous concentration results; $f=f(\X)$, $p^*=p/(p-1)$. For any norm $\|\cdot\|_a$, $\big\|\|f_k(\X)\|_{a}\big\|_{L_\infty}\coloneqq\left\|\!\left\|f_k(\X)\right\|_{a}(\X)\right\|_{L_{\infty}}$, similarly for $D_k$ (see \Cref{sec:Related Work}, \Cref{ass:McDiarmid's} for definitions). ``$t\asymp r$'' means a deviation of $\mathcal{O}(r)$ with probability $1-e^{-\Omega(K)}$\,.\textsuperscript{a}\\
  The last column indicates whether $f(\X)$ may have infinite variance.}\label{tab:comparison}}%
  {\footnotesize
  \setlength{\tabcolsep}{3pt}%
  \renewcommand{\arraystretch}{1.15}%
  \begin{threeparttable}
  \begin{tabular}{|c|lllcc|}
  \hline
  Setup & Result & Assumption & Deviation & Tail decay & $\infty$-var. \\
  \hline\hline
  & \tcell{\citeauthor{mcdiarmid1989method}\\ \citeyearpar{mcdiarmid1989method}}
    & $\big\|\|f_k(\X)\|_{L_\infty}\big\|_{L_\infty}<\infty$
    & $\Pp(f-\E f\!\ge\! t)$
    & $\exp(-\mathcal{O}(t^2))$
    & \nomark \\
  \thinline
  & \tcell{\citeauthor{maurer2021concentration}\\ \citeyearpar{maurer2021concentration}}
    & $\big\|\|f_k(\X)\|_{\psi_1}\big\|_{L_\infty}<\infty$
    & $\Pp(f-\E f\!\ge\! t)$
    & $\exp(-\mathcal{O}(t^2\wedge t))$
    & \nomark \\
  \thinline
  & \tcell{\citeauthor{li2024concentration}\\ \citeyearpar{li2024concentration}, Cor.~29}
    & \tcell{$\big\|\|f_k(\X)\|_{L_p}\big\|_{L_\infty}<\infty$,\\ $p\!\ge\!2$}
    & $\Pp(|f-\E f|\!\ge\! t)$
    & $\exp(-\mathcal{O}(t^2))+\mathcal{O}(t^{-p})$
    & \nomark \\
  \thinline
  & \tcell{\citeauthor{lei2026stability}\\ \citeyearpar{lei2026stability}, Thm.~2.6}
    & \tcell{$\big\|\|f_k(\X)\|_{L_p}\big\|_{L_\infty}<\infty$,\\ $p\in(1,2)$}
    & $\Pp(|f-\E f|\!\ge\! t)$
    & $\mathcal{O}(t^{-p})$
    & \yesmark \\
  \cline{2-6}
  \ours
  \cellcolor{white}\multirow{-9}{*}{General $f$}
  & \tcell{\textbf{Ours} -- \\ \Cref{thm:heavy-tail}}
    & \tcell{$\big\|\|\Dk\|_{L_p}\big\|_{L_\infty}<\infty$,\\ $p\in(1,2]$}
    & $\Pp(f-\E f\!\ge\! t)$
    & $\exp(-\mathcal{O}(t^{p^*}))$
    & \yesmark \\
  \hline
  \hline
  & \tcell{\citeauthor{wainwright2019high}\\ \citeyearpar{wainwright2019high}, Prop.~2.14}
    &  \tcell{indep., $X_k\le b$ a.s.,\\ $\E[(b-X_k)^2]\!<\!\infty$}
    & $\Pp(\overline X-\E\overline X\!\ge\! t)$
    & $\exp(-\mathcal{O}(n(t^2\wedge t)))$
    & \nomark \\
  \cline{2-6}
  \ours
  \cellcolor{white}\multirow{-2}{*}{\shortstack{Empirical Average\\ ($f=\overline X$)}}
  & \tcell{\citeauthor{aminiangeneralization} \citeyearpar{aminiangeneralization}; \\\textbf{Ours}\tnote{b} -- \Cref{col:gen_empirical}}
    & \tcell{indep., $X_k\le b$ a.s.,\\ $\E[(b-X_k)^p]\!<\!\infty$, $p\in(1,2]$}
    & $\Pp(\overline X-\E\overline X\!\ge\! t)$
    & $\exp(-\mathcal{O}(nt^{p^*}))$
    & \yesmark \\
  \hline
  & \tcell{\citeauthor{bubeck2013bandits}\\ \citeyearpar{bubeck2013bandits}, Lem.~2}
    & \tcell{i.i.d.,\\ $V_p\!\coloneqq\!\E[|X-\E X|^p]\!<\!\infty$}
    & $\Pp(\widehat\mu-\E X\!\ge\! t)$
    & $t\asymp V_p^{1/p}(K/n)^{1/p^*}$
    & \yesmark \\
  \cline{2-6}
  \ours
  \cellcolor{white}
  & \tcell{\textbf{Ours} -- \\ \Cref{col:mom_rate}}
    & \tcell{i.i.d., $X\le b$ a.s.,\\ $M_p\!\coloneqq\!\E[(b-X)^p]\!<\!\infty$}
    & $\Pp(\widehat\mu-\E X\!\ge\! t)$
    & \ccell{$t\asymp M_p^{1/p}(K/n)^{1/p^*}$}
    & \yesmark \\
  \thinline
  \ours
  \cellcolor{white}\multirow{-2}{*}{\shortstack{MoM\tnote{c}\\ ($f=\widehat\mu$)}}
  & \tcell{\textbf{Ours} -- \\ \Cref{rmk:mom}}
    & \tcell{i.i.d., $X\le b$ a.s.,\\ $M_p\!\coloneqq\!\E[(b-X)^p]\!<\!\infty$}
    & $\Pp(\widehat\mu-\E X\!\ge\! t)$
    & $\exp(-\mathcal{O}(nt^{p^*}))$
    & \yesmark \\
    \thinline
  \ours
  \cellcolor{white}
    & \tcell{\textbf{Ours} -- \\ \Cref{prop:non-iid-direct-mom}}&\tcell{indep., $X_{ji}\le b_{ji}$ a.s.,\\ $\E(b_{ji}-X_{ji})^p\!<\!\infty$}
    & $\Pp(\widehat\mu-\E\widehat\mu\!\ge\! t)$
    & $\exp(-\mathcal{O}(n(t/K)^{p^*}))$
    & \yesmark \\
  \hline
  \end{tabular}
  \begin{tablenotes}
     \item[a] That is, $\Pp(\widehat\mu-\E X\!\ge\! Cr)\le e^{-cK}$ for absolute constants $c,C>0$; choosing $K\asymp\ln(1/\delta)$ gives confidence
$1-\delta$.\\ This is a guarantee at the single scale $r$ fixed by $K$,
not a bound for every $t$, in contrast to the
$\exp(-\mathcal{O}(\,\cdot\,))$ entries. 
    \item[b] In fact, the upper bounds $b_k$ may vary across samples in our \Cref{col:gen_empirical}.
    \item[c] For the median-of-means (MoM) setup, $1<p\le2$.

  \end{tablenotes}
  \end{threeparttable}}
\end{table}

\section{Related Work}\label{sec:Related Work}
\paragraph{Generalizations of McDiarmid's inequality.} Let $X_1,\ldots,X_n$ be independent random variables taking values in $\mathcal X$, and write $\X\!=\!(X_1,\ldots, X_n)$. Let $P_{\X}\!=\!\bigotimes_{k=1}^n P_{X_k}$ be the law of $\X$. Let $f\!:\!\mathcal{X}^n\to\mathbb{R}$ be integrable with respect to $P_{\X}$. For every $k\!\in\![n]\!\coloneqq\!\{1,\dots,n\}$, let $X_k'$ have the same distribution as $X_k$ and be independent of $\X$. Let $\X_{-k}$ denote the random vector obtained by removing the $k$-th coordinate from $\X$. We define $\X^{(k)}\!\coloneqq\!(X_1,\ldots,X_{k-1},X_k',X_{k+1},\ldots,X_n).$

McDiarmid's inequality \citep{mcdiarmid1989method} assumes that replacing one coordinate changes the function's value by at most a fixed amount:
\[
|f(\X)-f(\X^{(k)})|\le c_k
\qquad\text{almost surely, for every }k,
\]
where $c_1,\ldots,c_n\!>\!0$ are deterministic constants. Under this bounded-differences condition, for every $t\!>\!0$,
\[
\mathbb P\!\left(f(\X)-\mathbb E f(\X)\ge t\right)
\le
\exp\!\left(-\frac{2t^2}{\sum_{k=1}^n c_k^2}\right).
\]
To extend its applicability in learning theory, numerous extensions relax the bounded-differences condition to accommodate unbounded functions \citep{kutin2002extensions, kutin2002almost, kontorovich2014concentration, maurer2021concentration, li2023distribution, li2024concentration_icml, li2024concentration, lei2026stability}. In particular, \citet{kontorovich2014concentration} replaces the finite-diameter assumption on the underlying metric probability spaces with a sub-Gaussian diameter condition. \citet{maurer2021concentration} develop related bounds under sub-Gaussian and sub-exponential assumptions on the fluctuations from varying one coordinate at a time. A key quantity introduced in \citet{maurer2021concentration}, later used in \citet{li2023distribution, li2024concentration, li2024concentration_icml}, is the $k$-th centered conditional version $f_k(\X)$ of $f$, a random-variable-valued function
\begin{equation*}
    f_k(\X)(x)\coloneqq f(x_1,\dots,x_{k-1},X_{k},x_{k+1},\dots,x_n)-\E_{P_{X_k}}[f(x_1,\dots,x_{k-1},X_{k},x_{k+1},\dots,x_n)].
\end{equation*}
Here, $f_k(\X)(x)$ measures the variation of $f$ in its $k$-th input, with randomness only from $X_k$ and the other coordinates fixed at $x_{-k}\!\coloneqq\!(x_1,\dots,x_{k-1},x_{k+1},\dots,x_n)$. For a norm $\|\cdot\|_a$,\footnote{Such as sub-Gaussian, sub-exponential, Orlicz or $L_p$ norms.} \citet{maurer2021concentration} also define the non-negative function $\|f_k(\X)\|_a(x)\!\coloneqq\!\|f_k(\X)(x)\|_a$ of $x\in\mathcal{X}^n$.

Prior works control this quantity: the assumption column of the General $f$ block in \Cref{tab:comparison} lists, for each result, the condition on $f_k(\X)$ implied by its assumptions. To compare the moment requirements, rewrite the $k$-th conditional variance $v_k(x)$ of $f$ at $x_{-k}\in\mathcal X^{n-1}$ as
\[
v_k(x)\coloneqq\operatorname{Var}_{X_k}\!\left(
f(x_1,\ldots,x_{k-1},X_k,x_{k+1},\ldots,x_n)
\right)=\E_{P_{X_k}}\big[f_k(\X)(x)^2\big] = \|f_k(\X)\|^2_{L_2}(x).
\]
The uniform sub-exponential assumptions in \citet{maurer2021concentration} imply the $k$-th conditional variance is bounded for every $k$ and $x\in\mathcal{X}^n$. The bounds of \citet{li2023distribution} instead use averaged variance quantities, whose finiteness requires $\mathbb E[v_k(\X)]<\infty$, together with additional tail and interaction assumptions. \citet{li2024concentration} derive concentration bounds from moment inequalities using Markov's inequality. Their explicit concentration results under polynomial moment assumptions, together with the corresponding results in \citet{li2024concentration_icml}, require uniformly bounded conditional variances and therefore finite variance of $f(\X)$.

\citet{lei2026stability} establish explicit McDiarmid-type tail bounds under finite $L_p$ moment conditions on coordinatewise differences. For $1<p<2$, these conditions allow $f(\X)$ to have infinite variance, but the resulting tail bounds decay polynomially for large deviations. We complement these results with a one-sided McDiarmid-type inequality that yields exponential tail bounds under assumptions allowing $f(\X)$ to have infinite variance.
\paragraph{Concentration for heavy-tailed random variables.}  \citet{lei2026stability} use a Student-$t$ example to illustrate that finite-moment conditions on increments alone cannot guarantee exponential concentration. This limitation is consistent with classical Fuk--Nagaev inequalities, which provide bounds for sums of independent random variables involving both exponential and polynomial terms \citep{FukNagaev1971, Fuk1973SomePI, Nagaev1979}. For regularly varying tails, large deviations of the empirical mean are dominated by a single extreme observation. This \emph{big-jump phenomenon} limits concentration to polynomial decay in the large-deviation regime \citep{Nagaev1979, mikosch1998large, bigjump2008, bakhshizadeh2023sharp}. However, \citet{aminiangeneralization} show that boundedness from above can yield exponential upper-tail concentration even when the variance is infinite. Specifically, they consider i.i.d.\ random variables satisfying $X_1\le b$ almost surely and $\|X_1\|_{L_p}<\infty$ for some $p\in(1,2]$. By writing $\overline X_n=n^{-1}\sum_{k=1}^n X_k$, they prove that
\[
\mathbb P\!\left(\overline X_n-\mathbb E X_1>t\right)
\le
\exp\!\left(
-\frac{p-1}{p}\,
\frac{nt^{p/(p-1)}}{\|X_1-b\|_{L_p}^{p/(p-1)}}
\right)
\]
for every $t>0$. For $p=2$, this bound becomes $\exp(-nt^2/(2\|X_1-b\|_{L_2}^2))$, yielding a sub-Gaussian upper-tail rate. Their result extends the classical one-sided Bernstein inequality \citep[Proposition~2.14]{wainwright2019high} to heavy-tailed distributions. We extend this result to general nonlinear functions of independent inputs by proving a McDiarmid-type inequality under a one-sided boundedness condition and suitable moment assumptions. Our bound provides exponential upper-tail concentration even when $f(\X)$ has infinite variance (\Cref{prop:McDiarmid's}). It recovers the empirical-mean bound of \citet{aminiangeneralization} with the same constants (\Cref{rmk:empirical}). We also apply it to the median-of-means (MoM) estimator, obtaining a concentration bound around its expectation.
\section{Results}\label{sec:results}
We now establish our one-sided McDiarmid-type inequality. Let $\X=(X_1,\ldots,X_n)$ be a vector of independent random variables taking values in $\mathcal X$, whose law is $P_{\X}=\bigotimes_{k=1}^n P_{X_k}$. Let $f:\mathcal X^n\to\mathbb R$ satisfy $\E[|f(\X)|]<\infty$. For each $k\in[n]\coloneqq\{1,\ldots,n\}$, let $\X_{-k}$ denote the vector obtained by removing the $k$-th coordinate from $\X$, and write $\E^{(k)}[\cdot]\coloneqq\E[\cdot\mid\X_{-k}]$.

Our assumption bounds the gap between $f$ and its essential supremum in each input. We define
\[
\sup_k f(x)
\coloneqq
\operatorname*{ess\,sup}_{z\sim P_{X_k}}
f(x_1,\ldots,x_{k-1},z,x_{k+1},\ldots,x_n),\quad x\in\mathcal X^n.
\]
Following the notation for $f_k(\X)(x)$ \citep{maurer2021concentration}, define
\[
D_k(\X)(x)
\coloneqq
\sup_k f(x)
-
f(x_1,\ldots,x_{k-1},X_k,x_{k+1},\ldots,x_n).
\]
For a deterministic $x\in\mathcal X^n$, this is a random variable whose randomness comes only from $X_k$; the remaining inputs are held fixed at $x_{-k}$. Evaluating at the random vector $\X$ gives
\[
D_k(\X)(\X)=\sup_k f(\X)-f(\X)\ge0
\qquad\text{almost surely}.
\]
This gap measures how much $f(\X)$ can increase when only its $k$-th input is changed. The upper bound $\sup_k f(\X)$ may depend on the remaining inputs $\X_{-k}$. For a norm $\|\cdot\|_a$,
\footnote{Such as sub-Gaussian, sub-exponential, Orlicz or $L_p$ norms.}
we write
\[
\|D_k(\X)\|_a(x)\coloneqq\|D_k(\X)(x)\|_a,\qquad x\in\mathcal X^n,
\]
where the norm is taken with respect to $P_{X_k}$ with $x$ held fixed. Since $D_k(\X)(x)$ depends on $x$ only through $x_{-k}$, so does $\|D_k(\X)\|_a(x)$.  Hence, $\|D_k(\X)\|_a(\X)$, the evaluation at $\X$, is a $\sigma(\X_{-k})$-measurable random variable. 
\begin{assumption}\label{ass:McDiarmid's}
There exists $p\in(1,2]$ such that
\begin{equation*}
V\coloneqq
\left\|
\sum_{k=1}^n
\|D_k(\X)\|_{L_p}^{p}(\X)
\right\|_{L_{\infty}(P_{\X})}
=
\left\|
\sum_{k=1}^n
\E^{(k)}\!\left[D_k(\X)(\X)^p\right]
\right\|_{L_{\infty}(P_{\X})}
<\infty.
\end{equation*}
\end{assumption}
Our assumption requires a uniform bound on the sum of the conditional $p$-th moments of these gaps.

\begin{theorem}\label{thm:heavy-tail}
Under \Cref{ass:McDiarmid's}, for every $t>0$,
\begin{equation}\label{eq:tail-p}
\mathbb P\!\left(f(\X)-\E[f(\X)]\ge t\right)
\le
\exp\!\left(
-\frac{p-1}{p}
\frac{t^{p/(p-1)}}{V^{1/(p-1)}}
\right).
\end{equation}
If $V=0$, then $f(\X)$ is almost surely constant and the deviation probability is zero.

\noindent Equivalently, for every $\delta\in(0,1)$, with probability at least $1-\delta$,
\begin{equation}\label{eq:confidence-heavy-tail}
f(\X)-\E[f(\X)]
\le
V^{1/p}
\left(
\frac{p}{p-1}\ln\frac1\delta
\right)^{(p-1)/p}.
\end{equation}
\end{theorem}

\Cref{ass:McDiarmid's} implies that $f(\X)$ is almost surely bounded above (\Cref{prop:bound}), but its lower tail may remain heavy. The following proposition makes the moment requirement precise.

\begin{proposition}\label{prop:McDiarmid's}
Under \Cref{ass:McDiarmid's}, $\E[|f(\X)|^p]<\infty$. However, for every $q>p$, there exist independent inputs and a function $f$ satisfying this assumption with $\E[|f(\X)|^q]=\infty$. Specifically, let $X_1,\ldots,X_n$ be i.i.d.\ Pareto random variables with scale parameter $1$ and shape parameter $q$, and set
$
f(\X)=-\frac1n\sum_{k=1}^n X_k.
$
Then \Cref{ass:McDiarmid's} holds with exponent $p$, while the $q$-th absolute moment of $f(\X)$, with $q>p$, is infinite. 
\end{proposition}
For $1<p<2$, taking $q=2$ in~\Cref{prop:McDiarmid's}, gives an example with infinite variance. Thus, the exponential upper-tail bound in \Cref{thm:heavy-tail} is compatible with an infinite second moment due to the lower tail.

Relative to \citet{aminiangeneralization}, \Cref{thm:heavy-tail} extends one-sided exponential concentration from i.i.d.\ empirical means to nonlinear functions of independent inputs, while recovering the same constants in the empirical-mean setting. Its moment assumptions allow infinite variance, beyond the finite-variance settings considered by \citet{maurer2021concentration, li2023distribution, li2024concentration_icml}. Compared with the polynomial tail bounds of \citet{lei2026stability}, our exponential upper-tail bound exploits the additional one-sided boundedness condition; finite-moment assumptions alone do not provide this guarantee. See \Cref{tab:comparison} for a comparison of the assumptions and concentration guarantees.
\begin{remark}\label{rmk:comparison}
    The General $f$ block in \Cref{tab:comparison} compares our assumption and results with other generalizations of McDiarmid's inequality. Existing works assume that $f_k(\X)(x)$ has a bounded norm. We instead assume that $D_k(\X)(x)$ has a bounded $L_p$-norm, which allows $f$ to have infinite variance while requiring $f$ to be bounded above. In exchange for this restriction, our bound on the upper tail decays exponentially in $t$. In contrast, the bound of \citet[Theorem 2.6]{lei2026stability}, the only existing result, to our knowledge, that allows $f$ to have infinite variance, decays only polynomially.
\end{remark}
\subsection{Reverse KL tensorization and proof of the main theorem}\label{sec:reverse_kl}
The key proof technique for \Cref{thm:heavy-tail} is a novel tensorization of the KL divergence (\Cref{thm:reverse-kl}), which interchanges the two arguments of the divergence relative to the classical entropy tensorization \citep[Theorem~6]{maurer2012thermodynamics}. Classical tensorization underlies the proofs of the sub-Gaussian and sub-exponential generalizations of McDiarmid's inequality \citep{maurer2021concentration} and of the Gaussian log-Sobolev inequality \citep{raginsky2013concentration}.

For an integrable function $h:\mathcal{X}^n\rightarrow\mathbb R$ satisfying $\mathbb E[e^{h(\X)}]<\infty$, we define the exponentially tilted probability measure $P_h$ on $\mathcal{X}^n$ by
\begin{equation}\label{eq:tilted_measure}
\frac{dP_h}{dP_{\X}}(x) \coloneqq \frac{e^{h(x)}}{\mathbb E[e^{h(\X)}]}.
\end{equation}
Although $P_{\X}$ is a product measure, $P_h$ generally is
not, since the function $h$ may introduce dependence between coordinates.

In classical concentration arguments based on the entropy method, one seeks to bound the quantity $\KL(P_{\gamma f}\Vert P_{\X})$ for $0\le\gamma\le\beta$ because it controls the centered log-moment generating function through the identity central to the Herbst method~\citep[Theorem~1]{maurer2012thermodynamics}
\begin{equation}\label{eq:Herbst}
  \ln\mathbb E\!\left[
    e^{\beta(f(\X)-\mathbb E[f(\X)])}
  \right]
  =
  \beta\int_0^\beta
  \frac{\KL(P_{\gamma f}\Vert P_{\X})}{\gamma^2}\,d\gamma.
\end{equation}
A bound on this quantity then yields a concentration inequality
via the Chernoff argument.

Entropy tensorization provides a way to bound this divergence on
the full product space using conditional divergences involving
one coordinate at a time \citep[Theorem~6]{maurer2012thermodynamics}.
By fixing the remaining coordinates, one can apply scalar
inequalities and combine the resulting estimates to obtain concentration for general, possibly nonlinear functions of independent random variables.
Under the assumptions stated below, we establish an inequality bounding $\KL(P_{\X}\Vert P_{\beta f})$ by the essential supremum of the sum of the conditional reverse divergences. This direction is particularly useful for concentration because
\[
  \ln\mathbb E[e^{h(\X)}]-\mathbb E[h(\X)]=\KL(P_{\X}\Vert P_h).
\]
Taking $h=\beta f$ therefore gives
\[
  \ln\mathbb E\!\left[
    e^{\beta(f(\X)-\mathbb E[f(\X)])}\right]=\KL(P_{\X}\Vert P_{\beta f}).
\]
Thus, reverse KL tensorization directly bounds the centered log-moment generating function, bypassing Herbst's argument. This can avoid losses
in constants introduced by intermediate entropy estimates, while retaining
the reduction to conditional quantities involving one coordinate at a time. A concentration inequality then follows from the Chernoff argument.

We apply this approach to heavy-tailed random variables. When $f$ is
bounded above, its conditional exponential moments can be controlled
using moments of order $p$, for $p\in(1,2]$,
of its distance from its conditional supremum. Uniform control of the sum of these conditional moments then yields an upper-tail
concentration inequality for $f(\X)$, even when its second moment is infinite. For $P_{\X_{-k}}$-a.e.\ $x_{-k}\in\mathcal{X}^{n-1}$, we define the tilted probability measure $P_{h,k}^{x_{-k}}$ on $\mathcal{X}$ by
\[\frac{dP_{h,k}^{x_{-k}}}{dP_{X_k}}(x_k)
  \coloneqq\frac{e^{h(x_k,x_{-k})}}
       {\E[e^{h(\X)}\mid\X_{-k}=x_{-k}]},\]
where $(x_k,x_{-k})$ denotes the point of $\mathcal{X}^n$ whose $k$-th coordinate is $x_k$ and whose remaining coordinates are given by $x_{-k}$.
\begin{proposition}\label{thm:reverse-kl}
Under the assumptions $h(\X)\le0\text{ a.s.}$ and
  $\mathbb E[|h(\X)|]<\infty$,
\begin{equation*}
  \KL(P_{\X}\Vert P_h)
  \le
  \esssup_{x\sim P_{\X}}
  \sum_{k=1}^n
  \KL\!\left(P_{X_k}\Vert P_{h,k}^{x_{-k}}\right).
\end{equation*}
Equivalently,
\begin{equation*}
  \ln\mathbb E[e^{h(\X)}]-\mathbb E[h(\X)]
  \le
  \left\|
    \sum_{k=1}^n
    \left(
      \ln\mathbb E^{(k)}[e^{h(\X)}]
      -\mathbb E^{(k)}[h(\X)]
    \right)
  \right\|_{L_{\infty}(P_{\X})}.
\end{equation*}
\end{proposition}
The proof decomposes the centered tilted density $dP_h/dP_{\X}-1$ by conditioning successively on the independent random variables and averaging the resulting decompositions over all possible orderings. We then apply the conditional Donsker--Varadhan inequality, using joint convexity and telescoping entropy differences to control the resulting KL terms. After canceling the forward KL contribution, the reverse divergence is bounded by an average of the sum of conditional reverse divergences under an auxiliary probability measure. Bounding this average by the essential supremum yields the result. 

\begin{remark}\label{rmk:kl_comparison}
A direct application of forward KL tensorization followed by the Herbst argument yields the same dependence on $t$, but a worse constant in the exponent for $1<p<2$. See Appendix~\ref{app:comparison_kl} for the derivation and a discussion of the limitations of this approach.
\end{remark}
We now prove \Cref{thm:heavy-tail} by combining reverse KL tensorization with a scalar inequality that controls conditional exponential moments using the $p$-th moments of $D_k(\X)$.
\begin{proofof}{\Cref{thm:heavy-tail}}
For $u\ge0$ and $1<p\le2$,
\begin{equation}\label{eq:scalar}
  e^{-u}\le1-u+\frac{u^p}{p}
\end{equation}
as shown in \citet[Lemma B.8]{behnamnia2025logsumexponential}. Indeed, the two sides are equal at $u=0$ and 
\[
\frac{d}{du}\left(1-u+\frac{u^p}{p}-e^{-u}\right)=-1+u^{p-1}+e^{-u}\ge0.
\]
For $0\le u\le1$, we use $u^{p-1}\ge u$ and $e^{-u}\ge1-u$. For $u\ge1$, we use $u^{p-1}\ge1$. 

Fix $\beta>0$. Since $f(\X)=\sup_kf(\X)-D_k(\X)(\X)$ and $\sup_kf(\X)$ does not depend on $X_k$,
\begin{align}
  \ln\mathbb E^{(k)}[e^{\beta f(\X)}]
  -\beta\mathbb E^{(k)}[f(\X)]
  &=
  \ln\mathbb E^{(k)}[e^{-\beta D_k(\X)(\X)}]
  +\beta\mathbb E^{(k)}[D_k(\X)(\X)]\notag\\
  &\le
  \frac{\beta^p}{p}
  \mathbb E^{(k)}[D_k(\X)(\X)^p]\text{ almost surely}.
  \label{eq:local-heavy-tail}
\end{align}
For the last inequality \eqref{eq:local-heavy-tail}, we apply \eqref{eq:scalar} with $u=\beta D_k(\X)(\X)\ge0$ a.s. and take expectations
\begin{equation*}
  \E^{(k)}\big[e^{-\beta D_k(\X)(\X)}\big]\le1-\beta\,\E^{(k)}[D_k(\X)(\X)]+\frac{\beta^p}{p}\,\E^{(k)}[D_k(\X)(\X)^p].
\end{equation*}
The left-hand side is positive, so the right-hand side is also positive. Since $\ln y\le y-1$ for $y>0$, 
\begin{align*}
    \ln\mathbb E^{(k)}[e^{-\beta D_k(\X)(\X)}]&\le\ln\left(1-\beta\,\E^{(k)}[D_k(\X)(\X)]+\frac{\beta^p}{p}\,\E^{(k)}[D_k(\X)(\X)^p]\right)\\
    &\le-\beta\,\E^{(k)}[D_k(\X)(\X)]+\frac{\beta^p}{p}\,\E^{(k)}[D_k(\X)(\X)^p].
\end{align*}
Adding $\beta\E^{(k)}[D_k(\X)(\X)]$ to both sides proves \eqref{eq:local-heavy-tail}. By \Cref{prop:bound}, there is a constant $b$ with $\beta (f-b)\le0$ a.s. Applying the reverse KL tensorization (\Cref{thm:reverse-kl}) to $h=\beta (f-b)$,
\begin{align*}
    \ln\mathbb E[e^{\beta(f(\X)-\mathbb E[f(\X)])}]&=
\ln\mathbb E[e^{\beta\left((f(\X)-b)-\mathbb E[f(\X)-b]\right)}]\\
  &\le\Big\|\sum_{k=1}^n\Big(\ln\E^{(k)}\big[e^{\beta f(\X)}\big]-\beta\,\E^{(k)}[f(\X)]\Big)\Big\|_{L_{\infty}(P_{\X})}\\
  &\le\frac{\beta^p}{p}\Big\|\sum_{k=1}^n\E^{(k)}\big[D_k(\X)(\X)^p\big]\Big\|_{L_{\infty}(P_{\X})}
  =\frac{V\beta^p}{p}.
\end{align*}
The last inequality is by \eqref{eq:local-heavy-tail}. Combined with Chernoff's inequality,
\begin{equation}\label{eq:reverse_chernoff}
  \mathbb P(f(\X)-\mathbb E[f(\X)]\ge t)
  \le
   \inf_{\beta>0}
  \exp\!\left(-\beta t+\ln\mathbb E[e^{\beta(f(\X)-\mathbb E[f(\X)])}]\right)
  \le
  \inf_{\beta>0}
  \exp\!\left(-\beta t+\frac{V\beta^p}{p}\right).
\end{equation}
If $V=0$, $f$ is almost surely constant, so the tail probability is zero. For $V>0$, taking $\beta=(t/V)^{1/(p-1)}$ proves \eqref{eq:tail-p}. The equivalent form \eqref{eq:confidence-heavy-tail} follows by solving for $t$.
\end{proofof}
\subsection{Applications}
We apply our concentration inequality (\Cref{thm:heavy-tail}) to provide deviation bounds for the empirical mean and the median-of-means estimator where the samples are independent, almost surely bounded above, and have finite $p$-th moments for some $p\in(1,2]$. 
\subsubsection{Empirical Mean Estimator}\label{sec:empirical_mean}
We first apply \Cref{thm:heavy-tail} to the empirical mean of independent observations that need not be identically distributed. Let $X_1,\ldots,X_n$ satisfy $X_k\le b_k$ almost surely for deterministic constants $b_k$, and suppose that
\[
M_{k,p}\coloneqq\mathbb E[(b_k-X_k)^p]<\infty
\]
for some $p\in(1,2]$. We define the empirical average $\overline X_n$ and the averaged moment $\overline M_n$ as 
\[
\overline X_n\coloneqq\frac1n\sum_{k=1}^n X_k,
\qquad
\overline M_n\coloneqq\frac1n\sum_{k=1}^n M_{k,p}.
\]
The next corollary bounds the upper-tail deviations of $\overline X_n$ in terms of the averaged moment $\overline M_n$.
\begin{corollary}\label{col:gen_empirical}
Assume $\overline M_n>0$. For every $t>0$, 
\begin{equation}\label{eq:mean-non-iid}
  \mathbb P\!\left(\overline X_n-\mathbb E[\overline X_n]\ge t\right)
  \le
  \exp\!\left(
    -\frac{p-1}{p}
    \frac{nt^{p/(p-1)}}{\overline M_n^{1/(p-1)}}
  \right).
\end{equation}
Equivalently, for every $\delta\in(0,1)$, with probability at least $1-\delta$,
\begin{equation}
\label{eq:non-iid-empirical-confidence}
  \overline X_n-\mathbb E[\overline X_n]
  \le \overline M_n^{1/p}
  \left(\frac{p}{p-1}\frac{\ln(1/\delta)}{n}\right)^{(p-1)/p}.
\end{equation}
\end{corollary}
\begin{remark}\label{rmk:empirical} If we assume that the samples are identically distributed in addition to their independence, i.e., $X_1,\ldots, X_n$ are i.i.d.\ copies of $X$, where $X$ is a random variable satisfying $X\le b$ a.s. and $0<M_p\coloneqq\mathbb E[(b-X)^p]<\infty$, our assumption coincides with the assumption of \citet[Theorem 4.3]{aminiangeneralization}. By substituting $\overline M_n=M_p$ and $\E[\overline X_n]=\E[X]$ into \Cref{col:gen_empirical}, we recover the same bound as in \citet[Theorem 4.3]{aminiangeneralization}. See \Cref{tab:comparison} for a comparison.
\end{remark}
\begin{remark}\label{rmk:empirical_comparison} 
\citet[Theorem 2.6]{lei2026stability} followed by Markov's inequality, gives, for every $t>0$,
\begin{equation}\label{eq:polynomial-optimized}
\Prob(\overline X_n-\E\overline X_n\ge t) \le \min\bigg\{1,\; \min_{Q\ge1} \bigg[ 2(4Q)^p\frac{\overline M_n}{t^pn^{(p-1)}} + 2\left( 2e\,4^pQ^{p-1}\frac{\overline M_n}{t^pn^{(p-1)}} \right)^Q \bigg]\bigg\}.
\end{equation}
Each step of the derivation is sharp (see Appendix \ref{app:comparison}.) Nevertheless, for every $t>0$, the right-hand side of
\eqref{eq:polynomial-optimized} is strictly larger than that of the one-sided exponential bound \eqref{eq:mean-non-iid} in \Cref{col:gen_empirical}. 
\end{remark}
\subsubsection{Median-of-Means Estimator}\label{sec:mom}
We next consider the median-of-means (MoM) estimator as a nonlinear application of our theorem. We study two distinct deviations: fluctuations around the estimator's own expectation, allowing non-identically distributed observations and unequal block sizes, and fluctuations around the population mean for i.i.d.\ observations partitioned into equal-sized blocks.

\noindent Let $K\ge3$ be odd, and let $m_1,\ldots,m_K$ be positive integers with $n=\sum_{j=1}^K m_j$. Consider mutually independent observations $\{X_{ji}:1\le j\le K,\ 1\le i\le m_j\}$ satisfying $X_{ji}\le b_{ji}$ almost surely for deterministic constants $b_{ji}$. Assume that, for some $p\in(1,2]$,
\[
M_{ji,p}\coloneqq\E[(b_{ji}-X_{ji})^p]<\infty.
\]
Define the block means and the MoM estimator by
\[
Y_j\coloneqq\frac1{m_j}\sum_{i=1}^{m_j}X_{ji},
\qquad
\widehat\mu_{\mathrm{MoM}}\coloneqq\operatorname{med}(Y_1,\ldots,Y_K).
\]

\paragraph{Concentration around the estimator's expectation.}
We first apply \Cref{thm:heavy-tail} directly to $\widehat\mu_{\mathrm{MoM}}$ as a function of all $n$ observations. The expectations $\E[Y_j]$ of the block means need not coincide, and the following result controls the estimator's fluctuations around $\E[\widehat\mu_{\mathrm{MoM}}]$. Bounds on the deviation probability $\Prob(\widehat\mu_{\mathrm{MoM}}-\E\widehat\mu_{\mathrm{MoM}}\ge t)$ for the MoM and other robust estimators are often needed in the analysis of stochastic optimization algorithms relying on \emph{biased} gradient estimators \citep{nazin2019algorithms, puchkin24a, Parletta2024}. The analysis of such algorithms typically reduces to controlling the bias of the gradient
estimator and its concentration around its own mean, each handled separately. We are not aware of any prior exponential bounds on $\Prob(\widehat\mu_{\mathrm{MoM}}-\E\widehat\mu_{\mathrm{MoM}}\ge t)$ in a heavy-tailed setting.
\begin{corollary}
\label{prop:non-iid-direct-mom}
Assume 
\[V_{\mathrm{mom}}\coloneqq\sum_{j=1}^K\frac{1}{m_j^p}\sum_{i=1}^{m_j}M_{ji,p}>0.\]
For every $t>0$,
\begin{equation}
\label{eq:non-iid-direct-mom}
  \Prob(\widehat\mu_{\mathrm{MoM}}-\E \widehat\mu_{\mathrm{MoM}}\ge t)
  \le
  \exp\!\left(
    -\frac{p-1}{p}\,
    \frac{t^{p/(p-1)}}{V_{\mathrm{mom}}^{1/(p-1)}}
  \right).
\end{equation}
\end{corollary}
\begin{remark}\label{rmk:mom_same_non_iid}
    Notice that if all the blocks have the same size $m_j=m$ in \Cref{prop:non-iid-direct-mom}, the deviation bound becomes
    \begin{equation}
\label{eq:equal-size-mom}
  \Prob(\widehat\mu_{\mathrm{MoM}}-\E \widehat\mu_{\mathrm{MoM}}\ge t)
  \le
  \exp\!\left(
    -\frac{p-1}{p}\,
    \frac{n}{\overline M_n^{1/(p-1)}}\left(\frac{t}{K}\right)^{p/(p-1)}
  \right),
\end{equation}
where $\overline M_n\coloneqq\frac 1n \sum_{j=1}^K\sum_{i=1}^mM_{ji,p}~(n=Km)$ is the averaged moment. 
\end{remark}
\paragraph{Concentration around the population mean.}
We then consider the case where the samples are i.i.d.\ copies of a random variable $X$ satisfying $X\le b$
a.s.\ and $0<M_p\coloneqq\E[(b-X)^p]<\infty$ for some $p\in(1,2]$. The samples are divided into $K$ blocks of the same size $m=n/K.$ We bound the deviation of $\widehat\mu_{\mathrm{MoM}}$ from the common population mean $\mu\coloneqq\E[X]$. This deviation is a quantity studied for the analysis of robust risk minimization \citep{MoMsurvey2020} or heavy-tailed bandits \citep{bubeck2013bandits}.
\begin{corollary}\label{col:mom_rate}
Let $X_1,\ldots,X_n$ be i.i.d.\ copies of $X$. For $t>\left(\frac{p}{p-1}\right)^{\frac{p-1}{p}} M_p^{1/p}\left(\frac{1}{m}\ln\frac{2K}{K+1}\right)^{\frac{p-1}{p}},$
\begin{equation*}
  \mathbb P\!\left(\widehat\mu_{\mathrm{MoM}}-\mu\ge t\right)
  \le
  \exp\!\left(-K d\left(\frac{K+1}{2K} \,\middle\Vert\, \overline q_t\right)\right),
\end{equation*}
where $d(\alpha\Vert \beta)\coloneqq \alpha\ln\frac{\alpha}{\beta}+(1-\alpha)\ln\frac{1-\alpha}{1-\beta}$ is the binary KL divergence, and $\overline q_t
  \coloneqq\exp\!\left(
    -\frac{p-1}{p}
    \frac{mt^{p/(p-1)}}{M_p^{1/(p-1)}}
  \right)$. If $K\ge\frac{\ln(1/\delta)}{\ln(2/\sqrt3)}$, with probability at least $1-\delta$,
\begin{equation}\label{eq:mom-confidence}
  \widehat\mu_{\mathrm{MoM}}-\mu
  \le
  M_p^{1/p}
  \left(
    \frac{p\ln4}{p-1}\frac{K}{n}
  \right)^{(p-1)/p}.
\end{equation}

\end{corollary}

\begin{remark}\label{rmk:mom_comp}
\citet[Lemma~2]{bubeck2013bandits} consider deviation bounds for heavy-tailed distributions. They obtain the deviation
$V_p^{1/p}(K/n)^{1/p^*}$ under the centered moment condition
$V_p\coloneqq\E[|X-\E X|^p]<\infty$. The high-probability form \eqref{eq:mom-confidence} attains the same deviation scale,
$M_p^{1/p}(K/n)^{1/p^*}$, with the centered moment $V_p$ replaced by the
one-sided moment $M_p$. See \Cref{tab:comparison} for a comparison.
\end{remark}
The proof of \Cref{col:mom_rate} rests on the identity
\begin{equation}\label{eq:p_eq}
\mathbb P(\widehat\mu_{\mathrm{MoM}}-\mu\ge t)
=\mathbb P \left(\left\{\#\{j\in[K]: Y_j-\mu\ge t\}\ge \frac{K+1}{2}\right\}\right),
\end{equation}
  which holds since the median exceeds $\mu+t$ exactly when at least $(K+1)/2$ of the block means do. We combine this with the Chernoff bound and use $\mathbb P(Y_j-\mu\ge t)\le\overline q_t$ by \Cref{col:gen_empirical} to prove \Cref{col:mom_rate}. 
  
  Using the identity \eqref{eq:p_eq}, we can also prove a concentration bound that holds for every $t>0$. Instead of the Chernoff bound, we use a union bound to prove \Cref{rmk:mom}.
\begin{corollary}\label{rmk:mom} Consider the same setting as \Cref{col:mom_rate}. For every $t>0$,
\begin{equation*}
\mathbb P(\widehat\mu_{\mathrm{MoM}}-\mu\ge t)
\le\binom{K}{(K+1)/2}\,\overline q_t^{\,\frac{K+1}{2}}=\binom{K}{(K+1)/2}\exp\left(-\frac{(p-1)(K+1)}{2pKM_p^{1/(p-1)}}nt^{p/(p-1)}\right).
\end{equation*}
\end{corollary}

\subsubsection{Uniform risk bounds for unbounded losses}
\label{sec:empirical-processes}

We next bound the population risk uniformly over a class of nonnegative,
possibly unbounded losses. Let $X=(X_1,\ldots,X_n)$ consist of i.i.d.\
observations taking values in $\mathcal{X}$, and let $\mathcal{H}$ be a
nonempty countable index set. For every $h\in\mathcal{H}$, let
$\ell_h:\mathcal{X}\to[0,\infty)$ be measurable, and define
\[
R(h):=\mathbb{E}[\ell_h(X_1)],
\qquad
\widehat R_n(h):=\frac1n\sum_{k=1}^n\ell_h(X_k).
\]
Suppose that, for some $p\in(1,2]$,
\begin{equation}
    \label{eq:ep-moment}
0<M_p:=\sup_{h\in\mathcal{H}}
\mathbb{E}[\ell_h(X_1)^p]<\infty.
\end{equation}
Writing $\mathcal{L}:=\{\ell_h:h\in\mathcal{H}\}$, define its expected
Rademacher complexity by
\[
\mathfrak{R}_n(\mathcal{L})
:=\mathbb{E}_{X,\varepsilon}\!\left[
\sup_{h\in\mathcal{H}}\frac1n\sum_{k=1}^n
\varepsilon_k\ell_h(X_k)\right],
\]
where $\varepsilon_1,\ldots,\varepsilon_n$ are independent Rademacher
random variables, independent of $X$.

\begin{corollary}[Uniform risk bound]\label{cor:empirical-processes}
Under the preceding assumptions, suppose that $\mathfrak{R}_n(\mathcal{L})<\infty$. For every $\delta\in(0,1)$, with probability at least $1-\delta$, for every $h\in\mathcal{H}$,
\begin{equation}\label{eq:ep-uniform-risk}
R(h)\le\widehat R_n(h)+2\mathfrak{R}_n(\mathcal{L})+M_p^{1/p}\left(\frac{p}{p-1}\frac{\ln(1/\delta)}{n}\right)^{(p-1)/p}.
\end{equation}
\end{corollary}
We make the following comments. First, \Cref{cor:empirical-processes} applies to possibly infinite classes and uses the ordinary empirical mean, without truncation or robustification. Its dependence on the confidence level is through $(\ln(1/\delta))^{(p-1)/p}$, even when $1<p<2$ and the losses have infinite variance. Clearly, Condition~\eqref{eq:ep-moment} alone does not ensure that $\mathfrak{R}_n(\mathcal{L})$ is finite or tends to zero. When it tends to zero, the bound yields (one-sided) uniform convergence. Second, for $p=2$, the confidence term is of the order of $\sqrt{M_2\ln(1/\delta)/n}$. This has the same Rademacher-complexity form and confidence rate as the classical bounded-loss bounds of \citet{Bartlett2002}, with a uniform second moment bound replacing a uniform bound on the losses. In particular, if $B_\infty:=\sup_{h\in\mathcal{H}}\|\ell_h(X_1)\|_{L_\infty}<\infty$, then $M_2^{1/2}\le B_\infty$, so~\eqref{eq:ep-uniform-risk} recovers the bounded-loss rate up to numerical constants. Third, uniform bounds for unbounded losses under finite-moment assumptions were obtained by \citet[Section~4]{Cortes2019} and \citet[Section~5.1]{Cortes2021} using relative-deviation arguments. More directly, \citet[Theorem~2.1]{Marchina2019} establishes concentration under an integrated-tail bound of Pareto type, with tail index $\alpha\in(1,2)$; a finite $\alpha$-th moment is not required. Our corollary instead assumes finite $p$-th moments, $p\in(1,2]$, and gives an explicit bound, even for $p=2$. The two results are therefore complementary, rather than one subsuming the other. Finally, \citet[Theorem~5.4]{aminiangeneralization} obtains a uniform risk bound for finite hypothesis classes via a union bound, whereas our corollary applies to possibly infinite classes with finite Rademacher complexity.
\section{Conclusion and Future Work}
We establish a one-sided McDiarmid-type concentration inequality for functions of independent random variables (\Cref{thm:heavy-tail}). We assume that the function is almost surely bounded above and that its deviations from the coordinatewise essential suprema have uniformly bounded conditional $p$-th moments, for some $1<p\le2$. Under these conditions, we obtain exponential upper-tail concentration. For $1<p<2$, this guarantee holds even when the function has infinite variance. The inequality is particularly well-suited to settings where a possibly nonlinear estimator is bounded on one side by construction. Promising applications include convergence guarantees for stochastic optimization with median-of-means gradient estimators, stability-based generalization bounds with heavy-tailed losses, and regret analysis in heavy-tailed bandits. Beyond the main result, we believe that our proof technique, the \emph{reverse KL tensorization} (\Cref{thm:reverse-kl}), may prove useful for establishing other types of concentration inequalities for nonlinear functions.

\bibliographystyle{plainnat}
\bibliography{reference}

\appendix

\section{Deferred proofs}

Here is the proof of the main propositions and lemmas from the paper body.
\paragraph{Notations.}
 We use $x=(x_1,\ldots,x_n)$ and $y=(y_1,\ldots,y_n)$ in $\mathcal{X}^n$ to denote elements in $\mathcal{X}^n$. We use $x_{-k}$ to denote the vector obtained by removing the $k$-th coordinate from $x$. For simplicity, we write $f(x_{-k},z)$ to denote $f(x_1,\ldots,x_{k-1},z,x_{k+1},\ldots,x_n)$. 
 
 $\X$ is a random vector whose law is $P_{\X}=\bigotimes_{k=1}^n P_{X_k}$ as in the main body. Since the law of $\X$ is fixed throughout the appendix, we write $D_k$ for the function $D_k(\X): \mathcal{X}^n \to \mathbb{R}$. Thus $D_k(x)$ and $D_k(\Y)$ denote $D_k(\X)(x)$ and $D_k(\X)(\Y)$ for $x \in \mathcal{X}^n$ and any random vector $\Y$ on $\mathcal{X}^n$.

 For a measurable $u:\mathcal X^n\to\mathbb R$ that is $P_{\X}$-integrable or nonnegative, define
\[
  \E^{(k)}[u](x)
  \coloneqq
  \int_{\mathcal X}
  u(x_1,\ldots,x_{k-1},z,x_{k+1},\ldots,x_n)
  \,dP_{X_k}(z).
\]
 Since $\E^{(k)}[u](x)$ depends on $x$ only through $x_{-k}$, we also regard it as a function on $\mathcal X^{n-1}$ and write $\E^{(k)}[u](x_{-k})$. Because the coordinates of $\X$ are independent, Fubini's theorem gives
$\E[u(\X)\mid \X_{-k}] = \E^{(k)}[u](\X_{-k})$ almost surely, for each $k\in[n]$.

 \subsection{Proofs for the main result (\texorpdfstring{\Cref{sec:results}}{Section 3})}
\begin{proposition}\label{prop:bound}
\Cref{ass:McDiarmid's}  implies that $\exists b\in\mathbb{R}~\text{s.t.}~f(\X)\le b~\text{a.s.}$
\end{proposition}
\begin{proofof}{\Cref{prop:bound}}
   Let $\Y=(Y_1,\ldots,Y_n)$ be an independent copy of $\X$. For $k=0,\ldots,n$, define $H^{(k)}\colon \mathcal X^n\times\mathcal X^n\to\mathcal X^n$ by
\[
H^{(k)}(x,y)
\coloneqq
(y_1,\ldots,y_k,x_{k+1},\ldots,x_n),
\qquad x,y\in\mathcal X^n .
\]
We note that $H^{(0)}(x,y)=x$ and $H^{(n)}(x,y)=y$.
 For simplicity, we write 
 \[s_k(x_{-k})\coloneqq
\operatorname*{ess\,sup}_{z\sim P_{X_k}} f(x_{-k},z).\] $s_k(x_{-k})$ is $\sup_k f(x)$ in the main body. We make the independence of $\sup_i f(x)$ and $x_k$ explicit. With this notation, we can rewrite $D_k(x)$ as $D_k(x)=s_k(x_{-k})-f(x)$.

For each $k$, the vectors $H^{(k-1)}(x,y)$ and $H^{(k)}(x,y)$
agree in every coordinate except the $k$-th coordinate. By writing $H^{(k)}_{-k}(x,y)$ to denote the elements of $H^{(k)}(x,y)$ except for the $k$-th element, $H^{(k)}_{-k}(x,y)=H^{(k-1)}_{-k}(x,y)$. Hence, for
almost every $(x,y)$, by the definition of $s_k$,
\begin{align*}
f(H^{(k-1)}(x,y))
&\leq s_k\!\left(H^{(k-1)}_{-k}(x,y)\right)=
s_k\!\left(H^{(k)}_{-k}(x,y)\right),\\
f(H^{(k)}(x,y))
&=
s_k\!\left(H^{(k)}_{-k}(x,y)\right)
-D_k(H^{(k)}(x,y)).
\end{align*}
By adding both sides,
\[
f(H^{(k-1)}(x,y))-f(H^{(k)}(x,y))
\leq D_k(H^{(k)}(x,y)).
\]
Summing over $k\in[n]$ gives
\[
f(x)-f(y)
\leq
\sum_{k=1}^n D_k(H^{(k)}(x,y)).
\]
Since $D_k\geq 0$ almost surely, taking the positive part yields, for almost every $(x,y)$,
\[
(f(x)-f(y))_+
\leq
\sum_{k=1}^n D_k(H^{(k)}(x,y)).
\]

We then define \[q_k(x_{-k})
\coloneqq
\mathbb E^{(k)}[D_k(\X)^p|\X_{-k}=x_{-k}].\] 
Each $q_j(x_{-j})$ is nonnegative. Also, by \Cref{ass:McDiarmid's}, 
$\sum_{j=1}^n q_j(\X_{-j})\leq V$ $P_{\X}\text{-a.s.}$. It follows that, for every $k$,
\[
q_k(\X_{-k})
\leq
\sum_{j=1}^n q_j(\X_{-j})
\leq V
\qquad P_{\X}\text{-a.s.}
\]
Thus, by Fubini's theorem, for $P_{\X}$-almost every
$x$,
\begin{align*}
\mathbb E_{\Y}\!\left[D_k(H^{(k)}(x,\Y))^p\right]&=\mathbb E_{Y_1,\ldots,Y_{k-1}}\E_{Y_k}
 \left[
 D_k(H^{(k)}(x,\Y))^p
 \right]\\
&=
\mathbb E_{Y_1,\ldots,Y_{k-1}}
 \left[
 q_k(Y_1,\ldots,Y_{k-1},x_{k+1},\ldots,x_n)
 \right]
\leq V.
\end{align*}
Minkowski's inequality and monotonicity of $L_p$ norm now yields
\[
\left\|(f(x)-f(\Y))_+\right\|_{L_p(P_{\Y})}
\leq
\sum_{k=1}^n
\left\|D_k(H^{(k)}(x,\Y))\right\|_{L_p(P_{\Y})}
\leq nV^{1/p}.
\]
Equivalently,
\[
\mathbb E_{\Y}\!\left[(f(x)-f(\Y))_+^p\right]
\leq n^pV.
\]
Since $f(\Y)$ is finite almost surely, there exists $M\in\mathbb R$
such that $\rho\coloneqq\mathbb P(f(\Y)\leq M)>0$.
On the event $\{f(\Y)\leq M\}$,
\[
(f(x)-f(\Y))_+\geq (f(x)-M)_+.
\]
It follows that, for $P_{\X}$-almost every $x$,
\begin{align*}
\rho(f(x)-M)_+^p
\leq
\mathbb E_{\Y}\!\left[(f(x)-f(\Y))_+^p\right]
\leq n^pV.
\end{align*}
Therefore,
\[
f(x)
\leq
M+n\left(\frac{V}{\rho}\right)^{1/p}
\]
for $P_{\X}$-almost every $x$. This proves that $f(\X)$ is essentially
bounded above.
\end{proofof}

\begin{proofof}{\Cref{prop:McDiarmid's}}
    Let $\Y=(Y_1,\ldots,Y_n)$ be an independent copy of $\X$. Define random vectors $H^{(k)}$, for $k=0,\dots,~n$, taking values in $\mathcal X^n$ as 
\[
H^{(k)}
\coloneqq H^{(k)}(\Y,\X) = 
(X_1,\ldots,X_k,Y_{k+1},\ldots,Y_n),
\qquad k=0,\ldots,n,
\]
so that $H^{(0)}=\Y$ and $H^{(n)}=\X$. Then, 
\[
f(\X)-f(\Y)
=
\sum_{k=1}^n\Delta_k,~~\text{where}~~
\Delta_k\coloneqq f(H^{(k)})-f(H^{(k-1)}).
\]

Fix $k$ and write the elements of $H^{(k)}$ except for the $k$-th element as 
\[
\U_k\coloneqq(X_1,\ldots,X_{k-1},Y_{k+1},\ldots,Y_n).
\]
Conditionally on $\U_k$, the variables $X_k$ and $Y_k$ are
independent and have the same distribution. Define a random variable
\[
s_k(\U_k)
\coloneqq
\operatorname*{ess\,sup}_{z\sim P_{X_k}} f(\U_k,z).
\]
With this notation, we can write $D_k(\U_k,X_k)$ as $D_k(\U_k,X_k)=s_k(\U_k)-f(\U_k,X_k)$.\\
Since both $f(\U_k,X_k)$ and $f(\U_k,Y_k)$ are at most $s_k(\U_k)$ almost surely, by triangle inequality,
\begin{align*}
|\Delta_k|
&=
|f(\U_k,X_k)-f(\U_k,Y_k)| \\
&\leq
\bigl(s_k(\U_k)-f(\U_k,X_k)\bigr)
+
\bigl(s_k(\U_k)-f(\U_k,Y_k)\bigr) \\
&=
D_k(\U_k,X_k)+D_k(\U_k,Y_k).
\end{align*}
Minkowski's inequality therefore gives
\begin{align*}
\left(
\mathbb E\!\left[|\Delta_k|^p\mid \U_k\right]
\right)^{1/p}
&\leq
\left(
\mathbb E\!\left[D_k(\U_k,X_k)^p\mid \U_k\right]
\right)^{1/p} +\left(
\mathbb E\!\left[D_k(\U_k,Y_k)^p\mid \U_k\right]
\right)^{1/p} \\
&=
2\left(
\E^{(k)}[D_k(\X)^p](\U_k)
\right)^{1/p}.
\end{align*}
Since $\U_k$ has the same distribution as $\X_{-k}$,
\begin{align*}
\left(
\mathbb E\!\left[|\Delta_k|^p\mid \U_k\right]
\right)^{1/p}=
\left(
\mathbb E\!\left[|\Delta_k|^p\mid \X_{-k}\right]
\right)^{1/p}
\leq
2\E^{(k)}[D_k(\X)^p\mid\X_{-k}]
^{1/p}
=2\E^{(k)}[D_k(\X)^p]
^{1/p}.
\end{align*}
By non-negativity, for every $k$,
\[
\E^{(k)}[D_k(\X)^p]
\leq
\sum_{j=1}^n \E^{(j)}[D_j(\X)^p]
\leq V
\qquad P_{\X}\text{-a.s.},
\]
Therefore, we obtain
\[
\|\Delta_k\|_{L_p(P_{\X}\otimes P_{\Y})}\leq 2V^{1/p}.
\]
Hence, by Minkowski's inequality,
\begin{equation}\label{eq:independent-copy-Lp}
\|f(\X)-f(\Y)\|_{L_p(P_{\X}\otimes P_{\Y})}
\leq
\sum_{k=1}^n\|\Delta_k\|_{L_p(P_{\X}\otimes P_{\Y})}
\leq
2nV^{1/p}.
\end{equation}

We now show that \eqref{eq:independent-copy-Lp} implies
$f(\X)\in L_p(P_{\X})$. Since $f(\Y)$ is finite almost surely, there exists
$M<\infty$ such that $\rho\coloneqq\mathbb P\bigl(|f(\Y)|\leq M\bigr)>0$.
On the event $\{|f(\Y)|\leq M\}$,
\[
|f(\X)-f(\Y)|
\geq
\bigl(|f(\X)|-M\bigr)_+.
\]
Using the independence of $\X$ and $\Y$, we obtain
\begin{align*}
\mathbb E|f(\X)-f(\Y)|^p
&\geq
\mathbb E\!\left[
\mathbf 1_{\{|f(\Y)|\leq M\}}
\bigl(|f(\X)|-M\bigr)_+^p
\right] \\
&=
\rho\,
\mathbb E\bigl(|f(\X)|-M\bigr)_+^p.
\end{align*}
The left-hand side is finite by
\eqref{eq:independent-copy-Lp}, so 
\[
\mathbb E\bigl(|f(\X)|-M\bigr)_+^p<\infty.
\]
Since $|f(\X)|\leq M+\bigl(|f(\X)|-M\bigr)_+$, it follows that 
\[
\mathbb E|f(\X)|^p<\infty.
\]
Thus $f(\X)\in L_p(P_{\X})$, and therefore for every $0<r\leq p$,
\[
f(\X)\in L_r(P_{\X}).
\]

Now we prove that $f(\X)$ can have an infinite variance. For $1<p<q$, let $X$ have the Pareto distribution with scale parameter one and shape parameter $q$, so
\[
  \mathbb P(X>t)=t^{-q},\qquad t\ge 1.
\]
Let  \(X_1,\ldots,X_n\) be i.i.d.\ copies of the Pareto random variable \(X\). We consider $f(\X)=-\frac{1}{n}\sum_{i=1}^n X_i$. Then, for any fixed \(x\in\mathcal X^n\), we have
\[
f(x_1,\ldots,x_{k-1},z,x_{k+1},\ldots,x_n)
=
-\frac{1}{n}
\left(
\sum_{i\ne k}x_i+z
\right).
\]
Since \(z\ge 1\) almost surely and the above expression is decreasing in \(z\),
\[
\begin{aligned}
\sup_k f(x)
&=
\operatorname*{ess\,sup}_{z\sim P_{X_k}}
f(x_1,\ldots,x_{k-1},z,x_{k+1},\ldots,x_n) \\
&=
-\frac{1}{n}
\left(
\sum_{i\ne k}x_i+1
\right).
\end{aligned}
\]
Consequently,
\[
\begin{aligned}
D_k(\mathbf X)
&=
\sup_k f(\mathbf X)-f(\mathbf X) \\
&=
-\frac{1}{n}
\left(
\sum_{i\ne k}X_i+1
\right)
+
\frac{1}{n}\sum_{i=1}^n X_i =
\frac{X_k-1}{n}.
\end{aligned}
\]
Therefore,
\[
D_k(\mathbf X)^p
=
\frac{(X_k-1)^p}{n^p},
\]
and hence
\[
\mathbb E^{(k)}
\left[
D_k(\mathbf X)^p
\right]
=
\frac{1}{n^p}
\mathbb E\left[(X-1)^p\right]\quad P_{\X}\text{-a.s.}
\]
Since \(p<q\), the Pareto random variable \(X\) has a finite \(p\)-th
moment, and therefore
\[
\mathbb E[(X-1)^p]<\infty.
\]
It follows that
\[
\begin{aligned}
V
=
\left\|
\sum_{k=1}^n
\mathbb E^{(k)}
\left[
D_k(\mathbf X)^p
\right]
\right\|_{L_{\infty}(P_{\mathbf X})} =
\left\|
\frac{n}{n^p}
\mathbb E[(X-1)^p]
\right\|_{L_{\infty}(P_{\mathbf X})} =
n^{1-p}\mathbb E[(X-1)^p]
<\infty.
\end{aligned}
\]
Thus \Cref{ass:McDiarmid's} is satisfied.

On the other hand,
\[
|f(\mathbf X)|
=
\frac{1}{n}\sum_{i=1}^n X_i
\ge
\frac{X_1}{n},
\]
and therefore
\[
|f(\mathbf X)|^q
\ge
\frac{X_1^q}{n^q}.
\]
Taking expectations gives
\[
\mathbb E\left[|f(\mathbf X)|^q\right]
\ge
\frac{1}{n^q}
\mathbb E[X_1^q]
=
\frac{1}{n^q}
\mathbb E[X^q]
.
\]
For a Pareto random variable with tail index \(q\),
$\mathbb E[X^r]<\infty\Longleftrightarrow r<q$, and hence $\mathbb E[X^q]=\infty$. Therefore,
\[
\mathbb E\left[|f(\mathbf X)|^q\right]=\infty.
\]
In particular, if $1<p<q\le 2$, 
\[
\mathbb E\left[f(\mathbf X)^2\right]=\infty.
\]
Hence, for \(1<p<q\le 2\), we obtain an example for which $V<\infty$ i.e. \Cref{ass:McDiarmid's} is satisfied, while simultaneously $f(\mathbf X)\notin L_q(P_{\X})$ and $f(\mathbf X)\notin L_2(P_{\X}).$
\end{proofof}
\begin{remark}

Additionally, conditional Jensen's inequality and
\eqref{eq:independent-copy-Lp} give
\begin{align*}
\|f(\X)-\mathbb Ef(\X)\|_{L_p(P_{\X})}
&=
\left\|
\mathbb E_{\Y}\!\left[f(\X)-f(\Y)\mid \X\right]
\right\|_{L_p(P_{\X})} \\
&\leq
\|f(\X)-f(\Y)\|_{L_p(P_{\X}\otimes P_{\Y})} \\
&\leq
2nV^{1/p}.
\end{align*}
\end{remark}

\subsection{Proofs for the reverse KL tensorization (\texorpdfstring{\Cref{thm:reverse-kl}}{Proposition 3.4} in \texorpdfstring{\Cref{sec:reverse_kl}}{Section 3.1})}\label{app:rKL_rigorous}
\begin{proofof}{\Cref{thm:reverse-kl}}
   For $S\subseteq[n]$, we write $X_S\coloneqq\{X_i:i\in S\}$.
   
   Define the projection $\Pi_S$ of integrable $\sigma(X_{[n]})$-measurable random variables onto $\sigma(X_S)$-measurable ones, and a $\sigma(X_S)$-measurable random variable $g_S$, by
\[
  \Pi_S\coloneqq\prod_{j\notin S}\mathbb E^{(j)},
  \qquad
  g_S\coloneqq\Pi_S\!\left(\frac{e^{h(\X)}}{Z_h}\right)=\mathbb E\!\left[\tfrac{e^{h(\X)}}{Z_h}\,\Big|\,X_S\right],
  \qquad
  \text{where}\quad Z_h\coloneqq\E[e^{h(\X)}]
  .
\]

Using the tower property, 
\begin{equation}\label{later_density}
  \mathbb E[g_S]
  =\mathbb E\!\left[\mathbb E\!\left[\tfrac{e^{h(\X)}}{Z_h}\,\Big|\,X_S\right]\right]
  =\mathbb E\!\left[\tfrac{e^{h(\X)}}{Z_h}\right]=1.
\end{equation}
Also, the random variable $g_S$ is nonnegative. Thus, $g_S$ is a probability density with respect to $P_{\X}$. Since $P_{\X}$ is a product measure,
\[
  \mathbb E^{(i)}\mathbb E^{(j)}=\mathbb E^{(j)}\mathbb E^{(i)}.
\]
Moreover,
\[
  \mathbb E^{(k)}\mathbb E^{(k)}
  =
  \mathbb E^{(k)},
\]
because for any measurable and integrable $u$, $\mathbb E^{(k)}[u]$ does not depend on $x_k$.  Therefore, the
operators $\mathbb E^{(1)},\ldots,\mathbb E^{(n)}$ are commuting
projections. Consequently, for $S\subseteq[n]$, the operator $\Pi_S=\prod_{j\notin S}\mathbb E^{(j)}$ is independent of the order of the factors.  
This leads to
\[
  \Pi_{[n]}=I,
  \qquad
  \Pi_\varnothing u=\mathbb E[u].
\]
In particular,
\begin{equation}\label{later_ent}
    g_{[n]}=\frac{e^{h(\X)}}{Z_h},
  \qquad
  g_\varnothing=1.
\end{equation}
The projections also satisfy
\[
  \Pi_S\Pi_T=\Pi_{S\cap T},
\]
because applying both operators integrates all the coordinates in $S^c\cup T^c=(S\cap T)^c$. If $k\notin S$, then the commutation of the coordinate projections gives
\begin{align}\label{later_eq}
  \mathbb E^{(k)}[g_{S\cup\{k\}}]=
  \mathbb E^{(k)}
  \Pi_{S\cup\{k\}}
  \left(\frac{e^{h(\X)}}{Z_h}\right)
  =
  \prod_{j\notin S}\mathbb E^{(j)}
  \left(\frac{e^{h(\X)}}{Z_h}\right)
  =
  \Pi_S\left(\frac{e^{h(\X)}}{Z_h}\right)
  =g_S.
\end{align}

We now record some identities for KL divergences. Write $\rho\coloneqq e^{h(\X)}/Z_h$, so $P_h=\rho\,P_{\X}$. We have 
\[
\mathbb E[\rho]=1,\quad\ln\rho=h(\X)-\ln Z_h,\quad\rho>0~\text{a.s.}
\]
By definition, $\KL(P_h\Vert P_{\X})=\mathbb E_{P_h}[\ln\frac{dP_h}{dP_{\X}}]=\mathbb E[\rho\ln\rho]$. 
Substituting $\ln\rho=h(\X)-\ln Z_h$,
\[
\KL(P_h\Vert P_{\X})=\mathbb E[\rho\,h(\X)]-\ln Z_h\,\mathbb E[\rho]=\mathbb E[\rho\,h(\X)]-\ln Z_h.
\]
By definition, $\KL(P_{\X}\Vert P_h)=\mathbb E_{P_{\X}}[\ln\frac{dP_{\X}}{dP_h}]=\E[\ln(1/\rho)]$. Again, by $\ln\rho=h(\X)-\ln Z_h$,
\[
\KL(P_{\X}\Vert P_h)=\mathbb E[\ln Z_h-h(\X)]=\ln Z_h-\mathbb E[h(\X)].
\]
Using these identities,
\begin{align} \notag
  \mathbb E
  \left[
    \left(\frac{e^{h(\X)}}{Z_h}-1\right)h(\X)
  \right]
  -\KL(P_h\Vert P_{\X})
  &=
  \E[\rho h(\X)-h(\X)]-(\mathbb E[\rho\,h(\X)]-\ln Z_h)\\ \notag
  &=
  \ln Z_h-\mathbb E[h(\X)]\\ \label{rough_key}
  &=
  \KL(P_{\X}\Vert P_h).
\end{align}

Our goal is to bound the reverse KL in \eqref{rough_key}. Our strategy is to first bound $\mathbb E
  \left[
    \left(\frac{e^{h(\X)}}{Z_h}-1\right)h(\X)
  \right]$. Then, we subtract the forward KL term $\KL(P_h\Vert P_{\X})$ from both sides to obtain a bound on the reverse KL $\KL(P_{\X}\Vert P_h)$.

The first step is to decompose $\mathbb E
  \left[
    \left(\frac{e^{h(\X)}}{Z_h}-1\right)h(\X)
  \right]$.
Fix a permutation $\pi$ of $\{1,\dots,n\}$. We define $T_i\coloneqq\{\pi(1),\dots,\pi(i)\}$, with $T_0\coloneqq\varnothing$. By conditioning successively on the independent random variables and using a telescoping sum,
\begin{equation*}
    \frac{e^{h(\X)}}{Z_h}-1
    =\E\left[\frac{e^{h(\X)}}{Z_h} \;\middle|\;X_{[n]}\right]-\E\left[\frac{e^{h(\X)}}{Z_h} \;\middle|\;X_{\varnothing}\right]
    =\sum_{i=1}^{n}\E\left[\frac{e^{h(\X)}}{Z_h} \;\middle|\;X_{T_i}\right]-\E\left[\frac{e^{h(\X)}}{Z_h} \;\middle|\;X_{T_{i-1}}\right].
\end{equation*}
Denote the set of indices preceding $k$ in $\pi$ by $S_k(\pi)$, so taking $k=\pi(i)$ gives $S_{\pi(i)}(\pi)=T_{i-1}$. Using this notation,
\begin{equation*}
    \frac{e^{h(\X)}}{Z_h}-1
    =\sum_{k=1}^{n}\E\left[\frac{e^{h(\X)}}{Z_h} \;\middle|\;X_{S_k(\pi)\cup\{k\}}\right]-\E\left[\frac{e^{h(\X)}}{Z_h} \;\middle|\;X_{S_k(\pi)}\right].
\end{equation*}
Now let $\pi$ be uniform on the symmetric group $\mathfrak{S}_n\coloneqq \{\text{all permutations of $\{1,\ldots,n\}$}\}$. Averaging over $\pi$ gives
\begin{equation}\label{arc_summand}
    \frac{e^{h(\X)}}{Z_h}-1
    =\sum_{k=1}^n\sum_{S\subseteq[n]\setminus\{k\}}\mathbb P \bigl(S_k(\pi)=S\bigr)
    \left(\E\left[\frac{e^{h(\X)}}{Z_h} \;\middle|\;X_{S\cup\{k\}}\right]-\E\left[\frac{e^{h(\X)}}{Z_h} \;\middle|\;X_{S}\right]\right).
\end{equation}
For fixed $k\in[n]$ and fixed $S\subseteq[n]\setminus\{k\}$,
\begin{align*}
    \mathbb P&\bigl(S_k(\pi)=S\bigr)\\
  &=\frac{\#\text{ all the permutations of }\{1,\dots,n\} \text{ such that the $(|S|+1)$-th element is $k$ and $S_k(\pi)=S$}}{\#\text{ all the permutations of }\{1,\dots,n\}}\\
  &=\frac{|S|!\,(n-1-|S|)!}{n!}
\end{align*}
where the probability is over the random permutation $\pi$. As a check, for fixed $k$,
\[
\sum_{S\subseteq[n]\setminus\{k\}}\mathbb P\bigl(S_k(\pi)=S\bigr)
=
\sum_{l=0}^{n-1}\sum_{S\subseteq[n]\setminus\{k\},~|S|=l}\frac{1}{n}\binom{n-1}{l}^{-1}=1.
\]
For $k\in[n]$ and $S\subseteq[n]\setminus\{k\}$, define
\[
  w_{k,S}\coloneqq\frac{|S|!(n-1-|S|)!}{n!}=\mathbb P\bigl(S_k(\pi)=S\bigr).
\]
To express the summands in \eqref{arc_summand}, we define the random variables
\begin{align*}
  \ell_k
  &\coloneqq
  \sum_{S\subseteq[n]\setminus\{k\}}  \mathbb P\bigl(S_k(\pi)=S\bigr)\mathbb E\!\left[\tfrac{e^{h(\X)}}{Z_h}\,\Big|\,X_S\right]
  =\sum_{S\subseteq[n]\setminus\{k\}}w_{k,S}g_S,\\
  b_k
  &\coloneqq
  \sum_{S\subseteq[n]\setminus\{k\}}  \mathbb P\bigl(S_k(\pi)=S\bigr)\mathbb E\!\left[\tfrac{e^{h(\X)}}{Z_h}\,\Big|\,X_{S\cup\{k\}}\right]
  =\sum_{S\subseteq[n]\setminus\{k\}}w_{k,S}g_{S\cup\{k\}},\\
  a_k&\coloneqq b_k-\ell_k.
\end{align*}
Using this notation, from \eqref{arc_summand}, we obtain
\begin{align} \label{eq:ak-sum}
  \sum_{k=1}^n a_k=\frac{e^{h(\X)}}{Z_h}-1.
\end{align}
Thus,
\begin{align*}
  \mathbb E
  \left[
    \left(\frac{e^{h(\X)}}{Z_h}-1\right)h(\X)
  \right]
  =\sum_{k=1}^n\mathbb E[(b_k-\ell_k)h(\X)].
\end{align*}

The next step is to apply the conditional Donsker--Varadhan inequality to bound $\mathbb E[a_kh(\X)]=\mathbb E[(b_k-\ell_k)h(\X)]$. Fix $x_{-k}$. We apply the Donsker--Varadhan variational formula (\Cref{lem:DV} in Appendix) to the $k$-th
coordinate, with reference measure $P_{X_k}$, test function $z\mapsto h(x_{-k},z)$, and density $\rho\coloneqq b_k/\ell_k$. To confirm that $\rho$ is admissible, we can verify
\begin{equation*}
    \mathbb E^{(k)}
  \left[\frac{b_k}{\ell_k}\right]
  =
  \frac{\mathbb E^{(k)}[b_k]}{\ell_k}
  =\frac{\ell_k}{\ell_k}
  =1
\end{equation*}
since $\ell_k$ is $\sigma(\X_{-k})$-measurable and \eqref{later_eq} implies $\mathbb E^{(k)}[b_k]=\ell_k$.
Applying the Donsker--Varadhan variational formula (\Cref{lem:DV} in Appendix) with these choices gives, $P_{\X_{-k}}$-a.s.\ ,
\begin{equation*}\label{rough_varadhan}
    \mathbb E^{(k)}\left[\frac{b_k}{\ell_k}h(\X)\right]
  -\mathbb E^{(k)}\left[\frac{b_k}{\ell_k}\ln\frac{b_k}{\ell_k}\right]
  \le\ln\mathbb E^{(k)}\big[e^{h(\X)}\big].
\end{equation*}
Subtracting $\mathbb E^{(k)}[h(\X)]$ from both sides gives, $P_{\X_{-k}}$-a.s.\ ,
\begin{align}
  \mathbb E^{(k)}
  \left[
    \left(\frac{b_k}{\ell_k}-1\right)h(\X)
  \right]
  &\le
  \mathbb E^{(k)}
  \left[
    \frac{b_k}{\ell_k}
    \ln\frac{b_k}{\ell_k}
  \right]
  +\ln\mathbb E^{(k)}[e^{h(\X)}]
  -\mathbb E^{(k)}[h(\X)].
  \label{eq:gibbs-fibre}
\end{align}
Multiplying \eqref{eq:gibbs-fibre} by $\ell_k$ gives, $P_{\X_{-k}}$-a.s.\ ,
\begin{align*}
  \ell_k
  \mathbb E^{(k)}
  \left[
    \left(\frac{b_k}{\ell_k}-1\right)h(\X)
  \right]\le
  \ell_k
  \mathbb E^{(k)}
  \left[
    \frac{b_k}{\ell_k}
    \ln\frac{b_k}{\ell_k}
  \right]
  +
  \ell_k
  \left(
    \ln\mathbb E^{(k)}[e^{h(\X)}]
    -\mathbb E^{(k)}[h(\X)]
  \right).
\end{align*}
Since $\ell_k$ is $\sigma(\X_{-k})$-measurable, it can be moved inside
$\mathbb E^{(k)}$.  Hence, $P_{\X_{-k}}$-a.s.\ ,
\begin{align*}
  \ell_k
  \mathbb E^{(k)}
  \left[
    \left(\frac{b_k}{\ell_k}-1\right)h(\X)
  \right]
  &=
  \mathbb E^{(k)}
  \left[
    \ell_k
    \left(\frac{b_k}{\ell_k}-1\right)h(\X)
  \right]\\
  &=
  \mathbb E^{(k)}
  \left[
    (b_k-\ell_k)h(\X)
  \right]\\
  &=
  \mathbb E^{(k)}[a_kh(\X)].
\end{align*}
Similarly, $P_{\X_{-k}}$-a.s.\ ,
\begin{align*}
  \ell_k
  \mathbb E^{(k)}
  \left[
    \frac{b_k}{\ell_k}
    \ln\frac{b_k}{\ell_k}
  \right]
  =
  \mathbb E^{(k)}
  \left[
    \ell_k\frac{b_k}{\ell_k}
    \ln\frac{b_k}{\ell_k}
  \right]=
  \mathbb E^{(k)}
  \left[
    b_k\ln\frac{b_k}{\ell_k}
  \right].
\end{align*}
Consequently, $P_{\X_{-k}}$-a.s.\ ,
\begin{align*}
  \mathbb E^{(k)}[a_kh(\X)]
  \le
  \mathbb E^{(k)}
  \left[
    b_k\ln\frac{b_k}{\ell_k}
  \right]+
  \ell_k
  \left(
    \ln\mathbb E^{(k)}[e^{h(\X)}]
    -\mathbb E^{(k)}[h(\X)]
  \right).
\end{align*}
Taking expectation with respect to the remaining variables $\X_{-k}$,
or equivalently taking the full expectation with respect to $P_{\X}$, gives
\begin{align*}
  \mathbb E[a_kh(\X)]
  &\le
  \mathbb E
  \left[
    b_k\ln\frac{b_k}{\ell_k}
  \right]+
  \mathbb E
  \left[
    \ell_k
    \left(
      \ln\mathbb E^{(k)}[e^{h(\X)}]
      -\mathbb E^{(k)}[h(\X)]
    \right)
  \right].
\end{align*}
By summing over $k$,
\begin{align}
 \mathbb E
  \left[
    \left(\frac{e^{h(\X)}}{Z_h}-1\right)h(\X)
  \right]
  \le
  \sum_{k=1}^n\KL(b_kP_{\X}\Vert\ell_kP_{\X})+
  \sum_{k=1}^n\mathbb E
  \left[
    \ell_k
    \left(
      \ln\mathbb E^{(k)}[e^{h(\X)}]
      -\mathbb E^{(k)}[h(\X)]
    \right)
  \right].\label{arc_dv}
\end{align}

We then bound the first term of \eqref{arc_dv}. Using the property \eqref{later_density} and the non-negativity of $g_S$, we write $g_SP_{\X}$ for the probability measure on $\mathcal X^n$ given by $(g_SP_{\X})(A)\coloneqq\int_A g_S\,dP_{\X}$, so that $\frac{d(g_SP_{\X})}{dP_{\X}}=g_S$. Since $e^{h(\X)}/Z_h>0$, each $g_S$ is strictly positive almost surely. Thus, each measure $g_SP_{\X}$ satisfies $P_{\X}\ll g_SP_{\X}$. This also defines $\ell_kP_{\X}$ and $b_kP_{\X}$ accordingly. We write $\ell_k P_{\X}$ for the probability measure on $\mathcal X^n$ given by $(\ell_k P_{\X})(A)\coloneqq\int_A \ell_k\,dP_{\X}$, so that $\frac{d(\ell_k P_{\X})}{dP_{\X}}=\ell_k$, and similarly for $b_k P_{\X}$. Using this notation, each summand in the first term of \eqref{arc_dv} is
\[
  \mathbb E
  \left[
    b_k\ln\frac{b_k}{\ell_k}
  \right]
  =
  \KL(b_kP_{\X}\Vert\ell_kP_{\X}).
\]
We bound this quantity by the joint convexity of the KL divergence,
\begin{equation}\label{eq:Ck}
  \KL(b_kP_{\X}\Vert\ell_kP_{\X})
  \le
  \sum_{S\subseteq[n]\setminus\{k\}}w_{k,S}
  \KL\!\left(g_{S\cup\{k\}}P_{\X}\Vert g_SP_{\X}\right).
\end{equation}
For $S\subseteq[n]\setminus\{k\}$, by applying the Pythagorean identity for conditional projections (\Cref{lem:pythagoras} in Appendix) with $\mathcal G=\sigma(X_S)$ and
$f=g_{S\cup\{k\}}$,
\begin{align}
  \KL\!\left(g_{S\cup\{k\}}P_{\X}\Vert g_SP_{\X}\right)
  &={\rm Ent}_{P_{\X}}(g_{S\cup\{k\}})
    -{\rm Ent}_{P_{\X}}(g_S),
  \label{eq:entropy-difference}
\end{align}
where the entropy ($\rm Ent$) is defined in \Cref{lem:pythagoras}.
Note that, we can apply \Cref{lem:pythagoras} because, by the tower property,
\[
  \mathbb E\!\left[g_{S\cup\{k\}}\mid X_S\right]
  =\mathbb E\!\left[\,\mathbb E\!\left[\tfrac{e^{h(\X)}}{Z_h}\,\Big|\,X_{S\cup\{k\}}\right]\Big|\,X_S\right]
  =\mathbb E\!\left[\tfrac{e^{h(\X)}}{Z_h}\,\Big|\,X_S\right]=g_S.
\]
(This is the same identity $\mathbb E^{(k)}[g_{S\cup\{k\}}]=g_S$ established earlier.) Since $h\le0$ a.s., $e^{h(\X)}/Z_h\le Z_h^{-1}$ a.s. Conditional expectation preserves the bound, so $0<g_S\le Z_h^{-1}$ a.s. In addition, the function $x\mapsto x|\ln x|$ is bounded on $(0,Z_h^{-1}]$. Therefore, $\mathbb E[|g_S\ln g_S|]<\infty$. If moreover $S\subseteq T$, then $\ln g_S$ is $\sigma(X_S)$-measurable and $g_S=\mathbb E[g_T\mid X_S]$, so $\mathbb E[g_T|\ln g_S|]=\mathbb E[g_S|\ln g_S|]<\infty$. Consequently, for $S\subseteq T$,
\[
  \KL(g_TP_{\X}\Vert g_SP_{\X})
  =\mathbb E[g_T\ln g_T]-\mathbb E[g_T\ln g_S]<\infty .
\]
Summing \eqref{eq:Ck} over $k$ and substituting \eqref{eq:entropy-difference},
\begin{align*}
    \sum_{k=1}^n\KL(b_kP_{\X}\Vert\ell_kP_{\X})
  &\le\sum_{k=1}^n\sum_{S\subseteq[n]\setminus\{k\}}w_{k,S}
     \Bigl[{\rm Ent}_{P_{\X}}(g_{S\cup\{k\}})-{\rm Ent}_{P_{\X}}(g_S)\Bigr]\\
     &=\mathbb E_\pi\!\left[\sum_{k=1}^n
    \Bigl({\rm Ent}_{P_{\X}}\bigl(g_{S_k(\pi)\cup\{k\}}\bigr)
         -{\rm Ent}_{P_{\X}}\bigl(g_{S_k(\pi)}\bigr)\Bigr)\right]\\
        &=\mathbb E_\pi\!\left[{\rm Ent}_{P_{\X}}(g_{[n]})-{\rm Ent}_{P_{\X}}(g_\varnothing)\right]\\
        &={\rm Ent}_{P_{\X}}(g_{[n]})-{\rm Ent}_{P_{\X}}(g_\varnothing).
\end{align*}
By \eqref{later_ent}, we have
\begin{align*}
    &{\rm Ent}_{P_{\X}}(g_{[n]})
    =\E\left[\frac{e^{h(\X)}}{Z_h}\ln\left(\frac{e^{h(\X)}}{Z_h}\right)\right]-\E\left[\frac{e^{h(\X)}}{Z_h}\right]\ln\left(\E\left[\frac{e^{h(\X)}}{Z_h}\right]\right)=\E\left[\frac{e^{h(\X)}}{Z_h}\ln\left(\frac{e^{h(\X)}}{Z_h}\right)\right]\\
    &\qquad\quad\quad\quad=\KL(P_h\Vert P_{\X})\\
    &{\rm Ent}_{P_{\X}}(g_\varnothing)=\E\left[\ln(1)\right]-\ln(1)=0.
\end{align*}
Therefore,
\begin{align}\label{eq:Ck-sum}
  \sum_{k=1}^n\KL(b_kP_{\X}\Vert\ell_kP_{\X})
  \le \KL(P_h\Vert P_{\X}).
\end{align}

We now bound the second term of $\eqref{arc_dv}$.
For $T\subseteq[n]$, define
\[
  c_T
  \coloneqq\frac{|T|!(n-|T|)!}{(n+1)!},
  \qquad
  \ell\coloneqq\sum_{T\subseteq[n]}c_Tg_T.
\]
For $S\subseteq[n]\setminus\{k\}$,
\begin{equation}\label{later_third}
    c_S+c_{S\cup\{k\}}=\frac{|S|!(n-|S|)!}{(n+1)!}+\frac{(|S|+1)!(n-|S|-1)!}{(n+1)!}
  =\frac{|S|!(n-1-|S|)!}{n!}
  =w_{k,S}.
\end{equation}
Hence, using \eqref{later_eq} and \eqref{later_third},
\begin{equation}\label{later_second}
  \mathbb E^{(k)}[\ell]=\sum_{S\subseteq[n]\setminus\{k\}}c_S\mathbb E^{(k)}g_S+c_{S\cup\{k\}}\mathbb E^{(k)}g_{S\cup\{k\}}=\sum_{S\subseteq[n]\setminus\{k\}}
  (c_S+c_{S\cup\{k\}})g_S=\sum_{S\subseteq[n]\setminus\{k\}}w_{k,S}g_S=\ell_k.
\end{equation}
Using this relationship,
\begin{align}\notag
  &\mathbb E
  \left[
    \ell_k
    \left(
      \ln\mathbb E^{(k)}[e^{h(\X)}]
      -\mathbb E^{(k)}[h(\X)]
    \right)
  \right]\\ \notag
  &=
  \mathbb E
  \left[
    \mathbb E^{(k)}[\ell]
    \left(
      \ln\mathbb E^{(k)}[e^{h(\X)}]
      -\mathbb E^{(k)}[h(\X)]
    \right)
  \right]\\ \notag
  &=
  \mathbb E
  \left[
    \mathbb E^{(k)}
    \left[
      \ell
      \left(
        \ln\mathbb E^{(k)}[e^{h(\X)}]
        -\mathbb E^{(k)}[h(\X)]
      \right)
    \right]
  \right]\\ \label{rough_lk}
  &=
  \mathbb E
  \left[
    \ell
    \left(
      \ln\mathbb E^{(k)}[e^{h(\X)}]
      -\mathbb E^{(k)}[h(\X)]
    \right)
  \right].
\end{align}
The second equality holds because the quantity $\ln\mathbb E^{(k)}[e^{h(\X)}]-\mathbb E^{(k)}[h(\X)]$ depends on $\X$ only through $\X_{-k}$. Since the sum has finitely many terms,
\begin{align*}
  &\sum_{k=1}^n
  \mathbb E
  \left[
    \ell
    \left(
      \ln\mathbb E^{(k)}[e^{h(\X)}]
      -\mathbb E^{(k)}[h(\X)]
    \right)
  \right]=
  \mathbb E
  \left[
    \ell
    \sum_{k=1}^n
    \left(
      \ln\mathbb E^{(k)}[e^{h(\X)}]
      -\mathbb E^{(k)}[h(\X)]
    \right)
  \right].
\end{align*}
Since $\E[\ell]=\E[\E^{(k)}[\ell]]=\E[\ell_k]=1$,
\begin{align} \notag
  \mathbb E
  \left[
    \ell
    \sum_{k=1}^n
    \left(
      \ln\mathbb E^{(k)}[e^{h(\X)}]
      -\mathbb E^{(k)}[h(\X)]
    \right)
  \right]&\le
  \mathbb E[\ell]
  \left\|
    \sum_{k=1}^n
    \left(
      \ln\mathbb E^{(k)}[e^{h(\X)}]
      -\mathbb E^{(k)}[h(\X)]
    \right)
  \right\|_{L_{\infty}(P_{\X})}\\ \label{rough_final}
  &=
  \left\|
    \sum_{k=1}^n
    \left(
      \ln\mathbb E^{(k)}[e^{h(\X)}]
      -\mathbb E^{(k)}[h(\X)]
    \right)
  \right\|_{L_{\infty}(P_{\X})}.
\end{align}
We have thus proved, by \eqref{arc_dv}, \eqref{eq:Ck-sum}, and \eqref{rough_final},
\begin{align}\label{rough_result}
  \mathbb E
  \left[
    \left(\frac{e^{h(\X)}}{Z_h}-1\right)h(\X)
  \right]
  &\le
  \KL(P_h\Vert P_{\X})+
  \left\|
    \sum_{k=1}^n
    \left(
      \ln\mathbb E^{(k)}[e^{h(\X)}]
      -\mathbb E^{(k)}[h(\X)]
    \right)
  \right\|_{L_{\infty}(P_{\X})}.
\end{align}
Using \eqref{rough_key},
\[
  \KL(P_{\X}\Vert P_h)
  \le
  \left\|
    \sum_{k=1}^n
    \left(
      \ln\mathbb E^{(k)}[e^{h(\X)}]
      -\mathbb E^{(k)}[h(\X)]
    \right)
  \right\|_{L_{\infty}(P_{\X})}.
\]
With 
\[
\KL\!\left(P_{X_k}\Vert P_{h,k}^{x_{-k}}\right)
  =\ln\mathbb E^{(k)}[e^{h(\X)}](x_{-k})
    -\mathbb E^{(k)}[h(\X)](x_{-k}),
\]
this proves the claim. The equivalent form follows since 
\begin{align*}
  \KL(P_{\X}\Vert P_h)
  &=\ln Z_h-\mathbb E[h(\X)].
\end{align*}

\end{proofof}

\subsection{Proofs for the application to the empirical average (\texorpdfstring{\Cref{sec:empirical_mean}}{Section 3.2.1})}
\begin{proofof}{\Cref{col:gen_empirical}}
For $f(\X)=\overline X_n$, we have $D_k(\X)(\X)\le(b_k-X_k)/n$, so independence gives
\[
V \le \frac1{n^p}\sum_{k=1}^n M_{k,p} = n^{1-p}\overline M_n.
\]
Applying \Cref{thm:heavy-tail} yields \eqref{eq:mean-non-iid}, and inverting the tail bound gives \eqref{eq:non-iid-empirical-confidence}.
\end{proofof}

\subsection{Proofs for the application to MoM (\texorpdfstring{\Cref{sec:mom}}{Section 3.2.2})}
\begin{proofof}{\Cref{prop:non-iid-direct-mom} and \Cref{rmk:mom_same_non_iid}}
Define the block-level upper bounds and the function $f$ as follows: $b_j\coloneqq\frac{1}{m_j}\sum_{i=1}^{m_j}b_{ji},$ $B\coloneqq\med(b_1,\ldots,b_K),$ $f\coloneqq\widehat\mu_{\mathrm{MoM}}-B$. By assumption, $Y_j\le b_j$ for every $j$. This implies $f\le0$ almost surely. The finite moment assumptions on the samples
also imply that $\widehat\mu_{\mathrm{MoM}}$ is integrable. We denote the essential supremum with respect to $X_{ji}$, with all other observations held fixed, by $\sup_{ji}f$. Increasing $X_{ji}$ by $u\ge0$ increases $Y_j$ by $u/m_j$. Since the median is $1$-Lipschitz in each
of its arguments, 
\[
0\le D_{ji}\coloneqq\sup_{ji}f-f \le\frac{b_{ji}-X_{ji}}{m_j} \quad \text{(almost surely)}.
\]
We write $\E^{(ji)}$ for the expectation over $X_{ji}$ with the other observations fixed. The independence of the samples gives the bound
\[
\sum_{j=1}^K\sum_{i=1}^{m_j}\E^{(ji)}[D_{ji}^{\,p}] \le \sum_{j=1}^K\frac{1}{m_j^p}\sum_{i=1}^{m_j}M_{ji,p} =V_{\mathrm{mom}} \quad \text{(almost surely)}.
\]
Applying \Cref{thm:heavy-tail} to $f=\widehat\mu_{\mathrm{MoM}}-B$ yields
\eqref{eq:non-iid-direct-mom} because $f-\E f=\widehat\mu_{\mathrm{MoM}}-\E \widehat\mu_{\mathrm{MoM}}$.

  If all blocks have the same size $m_j=m$, then $V_{\mathrm{mom}}=m^{-p}\sum_{j=1}^K\sum_{i=1}^m M_{ji,p}=\overline M_n K^p n^{1-p}$, where $\overline M_n\coloneqq\frac 1n \sum_{j=1}^K\sum_{i=1}^mM_{ji,p}$. Substituting this into \eqref{eq:non-iid-direct-mom} gives \eqref{eq:equal-size-mom}.
\end{proofof}

\begin{remark}
    The tail bound \eqref{eq:non-iid-direct-mom} has the same form as \eqref{eq:tail-p} with $V$ replaced by $V_{\mathrm{mom}}$. The high-probability version follows by replacing $V$ by $V_{\mathrm{mom}}$ in \eqref{eq:confidence-heavy-tail}.
    
    Consider the same setting as \Cref{prop:non-iid-direct-mom}. For every $\delta\in(0,1)$, with probability at least
$1-\delta$,
\begin{equation*}
  \widehat\mu_{\mathrm{MoM}}-\E \widehat\mu_{\mathrm{MoM}}
  \le V_{\mathrm{mom}}^{1/p}
  \left(\frac{p}{p-1}\ln\frac{1}{\delta}\right)^{(p-1)/p}.
\end{equation*}
\end{remark}

\begin{proofof}{\Cref{col:mom_rate}}
The proof of \Cref{col:mom_rate} relies on the next lemma.
\begin{lemma}\label{lem:mom}
Let $X_1,\ldots,X_n$ be i.i.d.\ copies of $X$. For $t>0$, define $\overline q_t\coloneqq\exp\!\left( -\frac{p-1}{p}\frac{mt^{p/(p-1)}}{M_p^{1/(p-1)}}\right).$\\
If $\overline q_t<\frac{K+1}{2K}$,
\begin{equation*}
  \mathbb P\!\left(\widehat\mu_{\mathrm{MoM}}-\E[X]\ge t\right)
  \le
  \exp\!\left(-K d((K+1)/2K \Vert\overline q_t)\right).
\end{equation*}
Also,
\begin{equation*}
  \mathbb P\!\left(
    \widehat\mu_{\mathrm{MoM}}-\E[X]
    \ge
    M_p^{1/p}
    \left(
      \frac{p\ln4}{p-1}\frac{K}{n}
    \right)^{(p-1)/p}
  \right)
  \le
  \left(\frac{\sqrt{3}}{2}\right)^K.
\end{equation*}
\end{lemma}
\begin{proofof}{\Cref{lem:mom}}
For $t>0$, define
\begin{align}
  \overline q_t
  &\coloneqq\exp\!\left(
    -\frac{p-1}{p}
    \frac{mt^{p/(p-1)}}{M_p^{1/(p-1)}}
  \right).
\end{align}
    Fix $j\in[K]$. Recall that $Y_j = \frac{1}{m} \sum_{i \in [m]} X_{ji}$. Since $X_{ji} \leq b$ a.s. for every $ji$, and the observations are independent, by \Cref{col:gen_empirical}, for every $t>0$ 
\begin{align*}
    \mathbb P(Y_j-\mu\ge t)
  &\le
    \exp
  \left(
    -\frac{p-1}{p}
    \frac{mt^{p/(p-1)}}{M_p^{1/(p-1)}}
  \right) = \overline q_t.
\end{align*}
Writing $k\coloneqq(K+1)/2$,
\begin{equation*}
\{\widehat\mu_{\mathrm{MoM}}-\mu\ge t\}
=\{\#\{j\in[K]: Y_j-\mu\ge t\}\ge k\}
=\bigcup_{\substack{S\subseteq[K]\\ |S|=k}}\bigcap_{j\in S}\{Y_j-\mu\ge t\}.
\end{equation*}
Hence, with Markov's inequality and by independence of the blocks,
\begin{equation}
    \mathbb P \left(\widehat\mu_{\mathrm{MoM}}-\mu \geq t\right) = \mathbb P \left(\sum_{j=1}^K\mathbf 1_{\{Y_j-\mu\ge t\}}\ge K \frac{K+1}{2K}\right)\le \inf_{\lambda>0}\exp
  \left(
    -\lambda K\frac{K+1}{2K}
  \right)\prod_{j=1}^K
  \mathbb E
  \left[
    e^{\lambda\mathbf 1_{\{Y_j-\mu\ge t\}}}
  \right].
\end{equation}
For every $j$,
\begin{align*}
  \mathbb E
  \left[
    e^{\lambda\mathbf 1_{\{Y_j-\mu\ge t\}}}
  \right]
  =
  1+
  (e^\lambda-1)
  \mathbb P(Y_j-\mu\ge t)
  \le
  1-\overline q_t+\overline q_te^\lambda.
\end{align*}
It follows that
\begin{align}
  \mathbb P
  \left(
    \widehat\mu_{\mathrm{MoM}}-\mu\ge t
  \right)
  \le
  \inf_{\lambda>0}
  \exp
  \left[
    K
    \left(
      -\lambda \frac{K+1}{2K}
      +
      \ln
      \left(
        1-\overline q_t+\overline q_te^\lambda
      \right)
    \right)
  \right].
  \label{eq:mom-chernoff}
\end{align}
Since $\overline q_t<\frac{K+1}{2K}$, the number
\[
  \lambda
  =
  \ln
  \left(
    \frac{\frac{K+1}{2K}(1-\overline q_t)}
         {\overline q_t(1-\frac{K+1}{2K})}
  \right)
\]
is positive. For this choice,
\begin{align*}
  1-\overline q_t+\overline q_te^\lambda
  =
  1-\overline q_t
  +
  \overline q_t
  \frac{\frac{K+1}{2K}(1-\overline q_t)}
       {\overline q_t(1-\frac{K+1}{2K})}
  =
  (1-\overline q_t)
  \left(
    1+\frac{\frac{K+1}{2K}}{1-\frac{K+1}{2K}}
  \right)
  =
  \frac{1-\overline q_t}{1-\frac{K+1}{2K}}.
\end{align*}
Substituting this value of $\lambda$ into the exponent in
\eqref{eq:mom-chernoff} gives
$-\lambda \frac{K+1}{2K}
  +
  \ln
  \left(
    1-\overline q_t+\overline q_te^\lambda
  \right)=
  -d(\frac{K+1}{2K}\Vert\overline q_t).
$
Therefore,
\[
  \mathbb P
  \left(
    \widehat\mu_{\mathrm{MoM}}-\mu\ge t
  \right)
  \le
  \exp
  \left(
    -K d((K+1)/2K \Vert\overline q_t)
  \right).
\]
    
    For the second part, fix some $\rho$ such that  $0<\rho<\frac{K+1}{2K}$.
    \begin{equation*}
        \exp\!\left(
    -\frac{p-1}{p}
    \frac{mt^{p/(p-1)}}{M_p^{1/(p-1)}}
  \right)=\rho
  \Leftrightarrow t=M_p^{1/p}
    \left(
      \frac{p\ln(1/\rho)}{p-1}\frac{K}{n}
    \right)^{(p-1)/p}
    \end{equation*}
    For $t^*=M_p^{1/p}
    \left(
      \frac{p\ln(1/\rho)}{p-1}\frac{K}{n}
    \right)^{(p-1)/p}$, $\overline q_{t^*}=\rho<\frac{K+1}{2K}$. Using the first part, 
    \begin{equation*}
  \mathbb P\!\left(\widehat\mu_{\mathrm{MoM}}-\mu\ge M_p^{1/p}
    \left(
      \frac{p\ln(1/\rho)}{p-1}\frac{K}{n}
    \right)^{(p-1)/p}\right)
  \le
  \exp\!\left(-K d((K+1)/2K \Vert\rho )\right).
\end{equation*}
The claim follows by taking $\rho=\frac{1}{4}$ and $d\left((K+1)/2K\Vert1/4\right)\ge d(1/2\Vert1/4)=\ln (2/\sqrt{3}).$
\end{proofof}
Using the first part of \Cref{lem:mom} and the equivalence 
\[\overline q_t<\frac{K+1}{2K} \Leftrightarrow t>\left(\frac{p}{p-1}\right)^{\frac{p-1}{p}} M_p^{1/p}\left(\frac{1}{m}\ln\frac{2K}{K+1}\right)^{\frac{p-1}{p}}\]
proves the first part of the claim in \Cref{col:mom_rate}. 
Using the second part of \Cref{lem:mom} and the equivalence of $(\sqrt{3}/2)^K\le\delta \Leftrightarrow K\ge\frac{\ln(1/\delta)}{\ln(2/\sqrt3)}$ proves the second part of the claim in \Cref{col:mom_rate}.
\end{proofof}

\begin{proofof}{\Cref{rmk:mom}}
    As in the proof of \Cref{col:mom_rate}, we use  $\mathbb P(Y_j-\mu\ge t)\le\overline q_t$ (\Cref{col:gen_empirical}) and the identity \eqref{eq:p_eq}. Instead of the Chernoff bound, we can use the union bound together with independence. For every $t>0$,
\begin{equation*}
\mathbb P\!(\widehat\mu_{\mathrm{MoM}}-\mu\ge t)
\!=\!\mathbb P\!\left(\!\bigcup_{\substack{S\subseteq[K]\\ |S|=\frac{K+1}{2}}}\bigcap_{j\in S}\{Y_j-\mu\ge t\}\!\right)
\!\le\!\sum_{\substack{S\subseteq[K]\\ |S|=\frac{K+1}{2}}}\!\prod_{j\in S}\!
\mathbb P(Y_j-\mu\ge t)
\!\le\!\binom{K}{(K+1)/2}\,\overline q_t^{\,\frac{K+1}{2}}.
\end{equation*}
\end{proofof}

\subsection{Proofs for the application to uniform risk bounds for unbounded losses (\texorpdfstring{\Cref{sec:empirical-processes}}{Section 3.2.3})}
\begin{proofof}{\Cref{cor:empirical-processes}}
Define
\[
T_n\coloneqq\sup_{h\in\mathcal{H}} \bigl(R(h)-\widehat R_n(h)\bigr).
\]
Countability ensures measurability, and for any fixed $h_0\in\mathcal{H}$,
\[
R(h_0)-\widehat R_n(h_0)\le T_n \le\sup_{h\in\mathcal{H}}R(h)\le M_p^{1/p}.
\]
Hence $T_n$ is integrable. We apply \Cref{thm:heavy-tail} to $f(X)=T_n$. For $k\in[n]$, let
\[
T_{n,k}:=\sup_{h\in\mathcal{H}} \left(R(h)-\frac1n\sum_{j\ne k}\ell_h(X_j)\right).
\]
Nonnegativity of the losses implies $\sup_k f(X)\le T_{n,k}$. For $\eta>0$, choose an $\X_{-k}$-measurable index $h_{k,\eta}$ such that
\[
R(h_{k,\eta})-\frac1n\sum_{j\ne k}\ell_{h_{k,\eta}}(X_j) \ge T_{n,k}-\eta.
\]
Such a choice is obtained by taking the first qualifying index in a fixed enumeration of $\mathcal{H}$. Evaluating $T_n$ at this index yields
\[
0\le D_k(X)(X)\le T_{n,k}-T_n \le\eta+\frac{\ell_{h_{k,\eta}}(X_k)}n \quad \text{(almost surely)}.
\]
Since $h_{k,\eta}$ depends only on $\X_{-k}$, independence and \eqref{eq:ep-moment} give $\mathbb{E}^{(k)}[\ell_{h_{k,\eta}}(X_k)^p]\le M_p$ almost surely. Conditional Minkowski's inequality therefore gives
\[
\bigl(\mathbb{E}^{(k)}[D_k(\X)(\X)^p]\bigr)^{1/p} \le\eta+\frac{M_p^{1/p}}n.
\]
Letting $\eta\downarrow0$ through a deterministic sequence and summing over $k$, we obtain
\[
V=\left\|\sum_{k=1}^n \mathbb{E}^{(k)}[D_k(\X)(\X)^p]\right\|_{L_\infty(P_X)} \le n^{1-p}M_p.
\]
\Cref{thm:heavy-tail} now implies that, with probability at least $1-\delta$,
\[
T_n\le\mathbb{E}[T_n] +M_p^{1/p} \left(\frac{p}{p-1}\frac{\ln(1/\delta)}n\right)^{(p-1)/p}.
\]
Finally, let $X'$ be an independent copy of $X$, independent of $\varepsilon$. Symmetrization yields
\[
\begin{aligned}
\mathbb{E}[T_n] &\le\mathbb{E}_{X,X'}\!\left[ \sup_{h\in\mathcal{H}}\frac1n\sum_{k=1}^n \bigl(\ell_h(X'_k)-\ell_h(X_k)\bigr)\right]\\
&=\mathbb{E}_{X,X',\varepsilon}\!\left[ \sup_{h\in\mathcal{H}}\frac1n\sum_{k=1}^n \varepsilon_k\bigl(\ell_h(X'_k)-\ell_h(X_k)\bigr)\right]\\
&\le2\mathfrak{R}_n(\mathcal{L}).
\end{aligned}
\]
The equality follows by exchanging each pair $(X_k,X'_k)$ according to the sign of $\varepsilon_k$, and the last inequality follows by splitting the supremum of the difference into two suprema. Combining the last two displays proves the claim.
\end{proofof}

\section{Detailed calculations for comparisons with existing works}
\subsection{Necessity of reverse KL tensorization (\texorpdfstring{\Cref{rmk:kl_comparison}}{Remark 3.5})}\label{app:comparison_kl}
To support the necessity of \emph{reverse KL tensorization} (\Cref{thm:reverse-kl}) in proving \Cref{thm:heavy-tail}, we show that a naive application of the entropy method \citep{maurer2012thermodynamics} provides a tail bound that is worse than \Cref{thm:heavy-tail}.
\begin{proposition}\label{prop:impossibility}
    By the Chernoff bound and the Herbst argument \eqref{eq:Herbst}, the tail probability is bounded by
    \begin{equation*}
   \mathbb{P}(f(\X)-\E[f(\X)]\ge t)\le\inf_{\beta>0} \exp \left(\beta\int_0^\beta
\frac{\KL\!\left(P_{\gamma f} \,\middle\|\, P_{\X}\right)}{\gamma^2}\,d\gamma -\beta t\right),
   \end{equation*}
   where $P_{tf}$ is the tilted measure defined in \eqref{eq:tilted_measure}. $\KL\!\left(P_{tf} \,\middle\|\, P_{\X}\right)$ is bounded by the subadditivity of entropy \citep[Theorem 4.22]{boucheron2013concentration} \citep[Proposition 4.1]{posta1997spectral} as 
   \begin{equation*}
\KL(P_{t f}\Vert P_{\X})
\le
\sum_{i=1}^n
\KL\!\left(
P_{t f,\,X_i\mid \X_{-i}}\,\middle\Vert\,P_{X_i}\,\middle|\,P_{t f,\,\X_{-i}}
\right)=\E_{P_{t f}}
\left[
\sum_{i=1}^n
\KL\!\left(P^{(i)}_{tf}(\cdot\mid \X_{-i})\,\middle\Vert\,P_{X_i}\right)
\right].
\end{equation*}
To achieve the same bound as \eqref{eq:reverse_chernoff} by bounding the conditional KL divergence of the form 
\[
\KL\!\left(P^{(i)}_{tf}(\cdot\mid x_{-i})\,\middle\Vert\,P_{X_i}\right)
\le
C_p t^p\,\E^{(i)}[D_i^p|\X_{-i}=x_{-i}],\]
we need $C_p=\frac{p-1}{p}$. However, under \Cref{ass:McDiarmid's}, $C_p=\frac{p-1}{p}$ is not valid uniformly over all \(p\in(1,2]\) and all nonnegative
random variables. 
\end{proposition}

\begin{proof}
    By the Chernoff bound and the Herbst argument \eqref{eq:Herbst},
    \begin{equation}\label{eq:ent_chernoff}
   \mathbb{P}(f(\X)-\E[f(\X)]\ge t)\le\inf_{\beta>0} \exp (\beta\int_0^\beta
\frac{\KL\!\left(P_{tf} \,\middle\|\, P_{\X}\right)}{t^2}\,dt -\beta t),
   \end{equation}
   where $P_{tf}$ is the tilted measure defined in \eqref{eq:tilted_measure}. For $i\in[n]$, we define the conditional counterpart of a tilted measure $P_h$ by tilting the $i$-th coordinate only. For fixed \(x_{-i}\in\mathcal{X}^n\), let \(P^{(i)}_h(\cdot\mid x_{-i})\)
be the probability measure on the \(i\)-th coordinate space $\mathcal{X}$ with
\begin{equation*}
\frac{dP^{(i)}_h(\cdot\mid x_{-i})}{dP_{X_i}}(x_i)
\coloneqq\frac{e^{h(x_i,x_{-i})}}{\E^{(i)}\bigl[e^{h}\bigr](x_{-i})} .
\end{equation*}
We write the KL divergence between conditional tilted measure and the marginal distribution as a function on $\mathcal{X}^{n-1}$ as
\begin{equation*}
\KLi{h}(x_{-i})
\coloneqq\KL\!\left(P^{(i)}_h(\cdot\mid x_{-i})\,\middle\Vert\,P_{X_i}\right).
\end{equation*}
For the conditional KL divergence
averaged against the marginal of \(P_h\),
\begin{equation*}
\KL\!\left(P_{h,\,X_i\mid \X_{-i}}\,\middle\Vert\,P_{X_i}\,\middle|\,P_{h,\,\X_{-i}}\right)
\coloneqq\int \KLi{h}(x_{-i})\,dP_{h,\,\X_{-i}}(x_{-i})
=\E_{P_h}\bigl[\KLi{h}(\X_{-i})\bigr].
\end{equation*}
By subadditivity of entropy \citep[Theorem 4.22]{boucheron2013concentration}, \citep[Proposition 4.1]{posta1997spectral}, also known as KL tensorization,
\begin{equation*}
\KL(P_{t f}\Vert P_{\X})
\le
\sum_{i=1}^n
\KL\!\left(
P_{t f,\,X_i\mid \X_{-i}}\,\middle\Vert\,P_{X_i}\,\middle|\,P_{t f,\,\X_{-i}}
\right)
=
\E_{P_{t f}}
\left[
\sum_{i=1}^n
\KLi{t f}
\right].
\end{equation*}
Denote a function $D_i$ on $\mathcal X^n$ defined as $D_i(x_1^n)=\sup_i f(x_1^n)-f(x_1^n)$,
\begin{align*}
    \KLi{tf}(x_{-i})&=\KL\!\left(P^{(i)}_{tf}(\cdot\mid x_{-i})\,\middle\Vert\,P_{X_i}\right)
    =\int \frac{e^{tf(x_i,x_{-i})}}{\E^{(i)}\bigl[e^{tf}\bigr](x_{-i})} \ln \left(\frac{e^{tf(x_i,x_{-i})}}{\E^{(i)}\bigl[e^{tf}\bigr](x_{-i})}\right) dP_{X_i}\\&
    =\frac{\E^{(i)}[tfe^{tf}](x_{-i})}{\E^{(i)}\bigl[e^{tf}\bigr](x_{-i})}-\ln \left(\E^{(i)}\bigl[e^{tf}\bigr](x_{-i})\right)\\ &=\frac{\E^{(i)}[-tD_ie^{-tD_i}](x_{-i})}{\E^{(i)}\bigl[e^{-tD_i}\bigr](x_{-i})}-\ln \left(\E^{(i)}\bigl[e^{-tD_i}\bigr](x_{-i})\right)\\&
    =-\frac{\E^{(i)}[tD_ie^{-tD_i}](x_{-i})}{\E^{(i)}\bigl[e^{-tD_i}\bigr](x_{-i})}-\ln \left(\E^{(i)}\bigl[e^{-tD_i}\bigr](x_{-i})\right)\\&
    \leq-\ln\left(1+\frac{\E^{(i)}[tD_ie^{-tD_i}](x_{-i})}{\E^{(i)}\bigl[e^{-tD_i}\bigr](x_{-i})}\right)-\ln \left(\E^{(i)}\bigl[e^{-tD_i}\bigr](x_{-i})\right)\\&
    =-\ln\left(\E^{(i)}\bigl[(1+tD_i)e^{-tD_i}\bigr](x_{-i})\right)\\&
    \leq -\E^{(i)}\bigl[\ln\left((1+tD_i)e^{-tD_i}\right)\bigr](x_{-i})\\&
    =-\E^{(i)}\bigl[\ln\left(1+tD_i\right)-tD_i\bigr](x_{-i})
    \le\frac{t^p\E^{(i)}\bigl[ D_i^p\bigr](x_{-i})}{p}.
\end{align*}
The last inequality holds because $x-\ln(1+x)\le\frac{x^p}{p}$ for $p\in(1,2]$ and $x\ge0$.\\
Thus,
\begin{align*}
&\KL(P_{t f}\Vert P_{\X})
\le
\frac{t^p}{p}\E_{P_{t f}}
\left[
\sum_{i=1}^n
\E^{(i)}\bigl[ D_i(\X)^p\bigr]
\right]\le\frac{t^p}{p}V,\\
&\beta\int_{0}^{\beta}\frac{\KL\!\left(P_{tf} \,\middle\|\, P_{\X}\right)}{t^2}\,dt\leq\frac{V\beta^p}{p(p-1)}.
\end{align*}
Combined with \eqref{eq:ent_chernoff},
\[
\mathbb P(f(\X)-\mathbb E[f(\X)]\ge t)
  \le
  \inf_{\beta>0}
  \exp\!\left(-\beta t+\frac{V\beta^p}{p(p-1)}\right)=\exp\!\left(-\frac{p-1}{p}\left(\frac{p-1}{V}\right)^{1/(p-1)}t^{p/(p-1)}\right).
\]
The bound is worse than \eqref{eq:reverse_chernoff}. To achieve the same bound as \eqref{eq:reverse_chernoff} by bounding the conditional KL divergence of the form $\KLi{-t D_i}(x_{-i})
\le
C_p t^p\,\E^{(i)}[D_i^p](x_{-i})$, we need $C_p=\frac{p-1}{p}$. However, $C_p=\frac{p-1}{p}$ is not valid uniformly over all \(p\in(1,2]\) and all nonnegative
random variables. 
For example, let
\[
p=\frac65,
\qquad
\mathbb{P}(X=0)=e^{-10},
\qquad
\mathbb{P}(X=10)=1-e^{-10}.
\]
Then, this random variable satisfies \Cref{ass:McDiarmid's} with $n=1$ and $f(X)=-X$. A direct calculation gives
\[
\KL(P_{-X}\Vert P)
=-\E_{P_{-X}}[X]-\ln\E[e^{-X}]
=
\frac{10}{2-e^{-10}}
-\ln(2-e^{-10})
>
5-\ln2
>
4.
\]
On the other hand,
\[
\frac{p-1}{p}\mathbb{E}X^p
=
\frac16(1-e^{-10})10^{6/5}
<
\frac{10}{3}.
\]
Thus $C_p = \frac{p-1}{p}$ fails already at $p = \frac65$, so this method cannot recover \eqref{eq:reverse_chernoff} uniformly over $p\in(1,2]$.
\end{proof}

\subsection{Comparison with polynomial decay (\texorpdfstring{\Cref{rmk:empirical_comparison}}{Remark 3.8})}\label{app:comparison}
We first assume $1<p\le2$. \citet[Theorem 2.6]{lei2026stability} states that a function
$f$ of independent inputs, with nonnegative replacement envelopes
satisfying
\[
|f(X)-f(X^{(k)})|
\le H_k(X_k,X_k'),
\qquad
\E H_k^p<\infty,
\]
obeys, for every $Q\ge1$ and $z>0$,
\begin{equation}\label{eq:theorem}
\Prob(|f(X)-\E f(X)|>z)
\le
\sum_{k=1}^n
\Prob\left(H_k>\frac{z}{4Q}\right)
+
2\left(
\frac{e\,4^pQ^{p-1}\sum_{k=1}^n\E H_k^p}{z^p}
\right)^Q.
\end{equation}
Here $X_k'$ is an independent copy of $X_k$, and $X^{(k)}$
denotes $X$ with its $k$-th coordinate replaced by $X_k'$.

Take
\[
f(X)=\overline X_n,
\qquad
H_k=\frac{|X_k-X_k'|}{n}.
\]
Write
\[
Y_k=b_k-X_k,
\qquad
Y_k'=b_k-X_k'.
\]
Both variables are nonnegative almost surely. Since
$|a-b|^p\le a^p+b^p$ for $a,b\ge0$,
\begin{equation}\label{eq:replacement-moments}
\E H_k^p
=
\frac{\E|Y_k-Y_k'|^p}{n^p}
\le
\frac{2M_{k,p}}{n^p},
\qquad
\sum_{k=1}^n\E H_k^p\le2V_n,
\end{equation}
where $V_n\coloneqq\frac1{n^p}\sum_{k=1}^n M_{k,p}
=n^{1-p}\overline M_n.$
Here, the constant two in these estimates is optimal.
Indeed, fix $c>0$ and let
\[
Y_k=cB_k,\qquad X_k=b_k-cB_k,
\]
where $B_1,\ldots,B_n$ are independent
$\operatorname{Bernoulli}(\varepsilon)$ random variables.
Let $B_k'$ be their independent copies. Then
\[
M_{k,p}=c^p\E[B_k^p]=\varepsilon c^p,
\qquad
\E H_k^p
=\frac{c^p\E[|B_k-B_k'|^p]}{n^p}
=
\frac{2\varepsilon(1-\varepsilon)c^p}{n^p}.
\]
Consequently,
\[
\frac{\E H_k^p}{2M_{k,p}/n^p}
=
\frac{\sum_{k=1}^n\E H_k^p}{2V_n}
=
1-\varepsilon
\longrightarrow1
\qquad(\varepsilon\downarrow0).
\]
Thus no universal constant smaller than two can replace
two in either estimate in \eqref{eq:replacement-moments}.

Next, Markov's inequality gives
\begin{equation}\label{eq:markov}
\sum_{k=1}^n
\Prob\left(H_k>\frac{z}{4Q}\right)
\le
\frac{(4Q)^p}{z^p}\sum_{k=1}^n\E H_k^p
\le
2(4Q)^p\frac{V_n}{z^p}.
\end{equation}
To verify the sharpness,
fix $z>0$ and $Q\ge1$, and denote $a=z/(4Q)$.
In the preceding construction $Y_k=cB_k$, choose
\[
c=na(1+\delta),\qquad \delta>0.
\]
Then $H_k$ equals $a(1+\delta)$ with probability
$2\varepsilon(1-\varepsilon)$ and is otherwise zero.
Therefore,
\[
\frac{\sum_{k=1}^n\Prob(H_k>a)}
     {a^{-p}\sum_{k=1}^n\E H_k^p}
     =\frac{2n\varepsilon(1-\varepsilon)}
     {a^{-p}na^p(1+\delta)^p2\varepsilon(1-\varepsilon)}
=
(1+\delta)^{-p}
\longrightarrow1
\qquad(\delta\downarrow0).
\]
Moreover,
\[
\frac{\sum_{k=1}^n\Prob(H_k>a)}
     {2V_n/a^p}
=
\frac{2n\varepsilon(1-\varepsilon)}
     {2a^{-p}n^{-p}nc^p\varepsilon}
=
(1-\varepsilon)(1+\delta)^{-p}
\longrightarrow1
\qquad(\varepsilon,\delta\downarrow0).
\]
Substituting \eqref{eq:replacement-moments} and
\eqref{eq:markov} into \eqref{eq:theorem} yields
\[
\Prob(|\overline X_n-\E\overline X_n|>z)
\le
2(4Q)^p\frac{V_n}{z^p}
+
2\left(
2e\,4^pQ^{p-1}\frac{V_n}{z^p}
\right)^Q.
\]
For every $0<z<t$,
\[
\{\overline X_n-\E\overline X_n\ge t\}
\subseteq
\{|\overline X_n-\E\overline X_n|>z\}.
\]
Letting $z\uparrow t$ therefore gives
\begin{equation}\label{eq:polynomial}
\Prob(\overline X_n-\E\overline X_n\ge t)
\le
\min\left\{1,\;
2(4Q)^p\frac{V_n}{t^p}
+
2\left(
2e\,4^pQ^{p-1}\frac{V_n}{t^p}
\right)^Q
\right\}.
\end{equation}
Since \eqref{eq:polynomial} holds for every $Q\ge1$, we may
optimize over $Q$. The expression inside the minimum is
continuous in $Q$ and tends to infinity as $Q\to\infty$,
because $V_n>0$. Its minimum is therefore attained,
establishing \eqref{eq:polynomial-optimized}.
Let $r=t^p/V_n>0$. The exponential bound \eqref{eq:mean-non-iid} is
\[
E_p(r)\coloneqq\exp\!\left(-\frac{p-1}{p}r^{1/(p-1)}\right).
\]
A direct maximization shows
\[
\sup_{r>0}rE_p(r)=\left(\frac pe\right)^{p-1}<1
\text{ for } 1<p\le2.
\]
Thus $E_p(r)<\min\{1,1/r\}$. On the other hand, for every $Q\ge1$,
the right-hand side of \eqref{eq:polynomial} is 
\[
\min\left\{1,\;
2(4Q)^p\frac{V_n}{t^p}
+2\left(2e\,4^pQ^{p-1}\frac{V_n}{t^p}\right)^Q\right\}=\min\left\{1,\;
\frac{2(4Q)^p}{r}
+2\left(\frac{2e\,4^pQ^{p-1}}{r}\right)^Q\right\}\ge\min\left\{1,\frac{2\,4^p}{r}\right\}.
\]
Since $\min\left\{1,\frac{2\,4^p}{r}\right\}\ge\min\{1,1/r\}>E_p(r)$, \eqref{eq:mean-non-iid} is strictly smaller than the
moment-only bound \eqref{eq:polynomial} for every $t>0$, even after
optimizing over $Q\ge1$.

\section{Mathematical Results}
\begin{lemma}[Pythagorean identity for conditional projections]
\label{lem:pythagoras}
Let $f\ge0$ with $\mathbb E[f]=1$, let $\mathcal G\subseteq\mathcal F$ be a
sub-$\sigma$-algebra, and set $f_{\mathcal G}\coloneqq\mathbb E[f\mid\mathcal G]$.
For a probability measure $\mu$ and a measurable $f\ge0$ with
$\mathbb E_\mu[f]<\infty$, define the entropy of $f$ relative to $\mu$ by
\begin{equation*}
  {\rm Ent}_\mu(f)
  \coloneqq\mathbb E_\mu[f\ln f]-\mathbb E_\mu[f]\ln\mathbb E_\mu[f],
\end{equation*}
with the convention $0\ln0=0$. If $\mathbb E[f|\ln f|]<\infty$, then $fP_{\X}\ll f_{\mathcal G}P_{\X}$ and
\[
  \KL\!\left(fP_{\X}\,\Vert\,f_{\mathcal G}P_{\X}\right)
  ={\rm Ent}_{P_{\X}}(f)-{\rm Ent}_{P_{\X}}(f_{\mathcal G}).
\]

\end{lemma}

\begin{proof} 
Both $f$ and $f_{\mathcal G}$ are probability densities with respect to $P_{\X}$,
so ${\rm Ent}_{P_{\X}}(f)=\mathbb E[f\ln f]$ and likewise for $f_{\mathcal G}$.
Since $f_{\mathcal G}$ is $\mathcal G$-measurable, the set
$N\coloneqq\{f_{\mathcal G}=0\}$ lies in $\mathcal G$, so it is an admissible test set
in the defining property of conditional expectation,
$\int_A f\,d\mu=\int_A\mathbb E[f\mid\mathcal G]\,d\mu$ for all $A\in\mathcal G$.
Taking $A=N$ gives
\[
  \int_N f\,d\mu
  = \int_N \mathbb E[f\mid\mathcal G]\,d\mu
  =\mathbb E\big[\mathbf 1_N\,\mathbb E[f\mid\mathcal G]\big]
  =\mathbb E\big[\mathbf 1_N f_{\mathcal G}\big]=0 ,
\]
because $f_{\mathcal G}$ vanishes on $N$.
As $f\ge0$, this forces $f=0$ $\mu$-a.e.\ on $N$. Consequently, if
$A\in\mathcal F$ satisfies $\int_A f_{\mathcal G}\,d\mu=0$, then
$f_{\mathcal G}=0$ $\mu$-a.e.\ on $A$, i.e.\ $\mu(A\setminus N)=0$, and hence
$\int_A f\,d\mu=0$. Therefore $f\mu\ll f_{\mathcal G}\mu$, and for every
$A\in\mathcal F$,
\[
  \int_A\frac{f}{f_{\mathcal G}}\mathbf 1_{\{f_{\mathcal G}>0\}}f_{\mathcal G}\,d\mu
  =\int_{A\setminus N}f\,d\mu=\int_A f\,d\mu,
\]
so that
$\frac{d(f\mu)}{d(f_{\mathcal G}\mu)}=f/f_{\mathcal G}$ holds
$f_{\mathcal G}\mu$-a.e., the values on $N$ being irrelevant. Thus
\[
  \KL\!\left(f\mu\Vert f_{\mathcal G}\mu\right)
  =\mathbb E_\mu\!\left[f\ln\frac{f}{f_{\mathcal G}}\right],
  \qquad 0\ln 0\coloneqq0 .
\]
and
\[
  \KL\!\left(fP_{\X}\Vert f_{\mathcal G}P_{\X}\right)
  =\mathbb E\!\left[f\ln\frac{f}{f_{\mathcal G}}\right]
  =\mathbb E[f\ln f]-\mathbb E\!\left[f\ln f_{\mathcal G}\right].
\]
For the second term, note that $(\ln f_{\mathcal G})^{\pm}$ is
$\mathcal G$-measurable and nonnegative, so pulling it out of the conditional
expectation (valid for nonnegative factors, without integrability assumptions)
and using the tower property gives
\[
  \mathbb E\!\left[f(\ln f_{\mathcal G})^{\pm}\right]
  =\mathbb E\!\left[\mathbb E\!\left[f(\ln f_{\mathcal G})^{\pm}\mid\mathcal G\right]\right]
  =\mathbb E\!\left[(\ln f_{\mathcal G})^{\pm}\,\mathbb E[f\mid\mathcal G]\right]
  =\mathbb E\!\left[f_{\mathcal G}(\ln f_{\mathcal G})^{\pm}\right].
\]
$\mathbb E[f_{\mathcal G}\ln f_{\mathcal G}]
\le\mathbb E[f\ln f]<\infty$ by Jensen applied to the convex function
$x\mapsto x\ln x$. Also, $\mathbb E[-f_{\mathcal G}\ln f_{\mathcal G}]\le\frac{1}{e}$ since $\forall x\ge0$, $x\ln x\ge-\frac1e$. Subtracting,
$\mathbb E[f\ln f_{\mathcal G}]=\mathbb E[f_{\mathcal G}\ln f_{\mathcal G}]
={\rm Ent}_{P_{\X}}(f_{\mathcal G})$, which proves the claim.
\end{proof}

\begin{lemma}[Donsker--Varadhan variational formula]\label{lem:DV}
  Let $\mu$ be a probability measure and $f$ a measurable function with
  $\mathbb E_\mu[e^{f}]<\infty$. Then
  \begin{equation*}
    \ln\mathbb E_\mu\!\left[e^{f}\right]
    =\sup_{\nu\ll\mu}\Bigl\{\mathbb E_\nu[f]-\KL(\nu\Vert\mu)\Bigr\},
  \end{equation*}
  the supremum being attained at the tilted measure
  $d\tilde\mu\coloneqq e^{f}\,d\mu\big/\mathbb E_\mu[e^{f}]$.
  Equivalently, for every measurable $\rho\ge0$ with $\mathbb E_\mu[\rho]=1$,
  \begin{equation*}
    \mathbb E_\mu[\rho f]-\mathbb E_\mu[\rho\ln\rho]
    \;\le\;\ln\mathbb E_\mu\!\left[e^{f}\right].
  \end{equation*}
\end{lemma}
For a proof, see \citet[Lemma~4.10]{van2014probability}.

\end{document}